\documentclass{amsart}

\usepackage{amsmath,amssymb,amscd}
\usepackage{mathdots,mathabx}
\usepackage{blkarray}
\usepackage{caption}
\usepackage{hyperref}
\usepackage[all]{xy}
\usepackage{float}
\usepackage{mathtools}
\input{quivers_KC.sty}
\input{quivers_KC_triangles.sty}

\newtheorem{theorem}{Theorem}[section]
\newtheorem{lemma}[theorem]{Lemma}
\newtheorem{corollary}[theorem]{Corollary}
\newtheorem{proposition}[theorem]{Proposition}

\theoremstyle{definition}
\newtheorem{definition}[theorem]{Definition}
\newtheorem{example}[theorem]{Example}

\theoremstyle{remark}
\newtheorem{remark}[theorem]{Remark}

\numberwithin{equation}{section}

\DeclareMathOperator{\A}{A}

\DeclareMathOperator{\Coker}{Coker}

\DeclareMathOperator{\GL}{GL}
\DeclareMathOperator{\Gr}{Gr}
\DeclareMathOperator{\Hom}{Hom}
\DeclareMathOperator{\Id}{Id}

\DeclareMathOperator{\Ker}{Ker}
\DeclareMathOperator{\Mod}{Mod}

\DeclareMathOperator{\PHom}{PHom}

\DeclareMathOperator{\Rep}{Rep}

\DeclareMathOperator{\SI}{SI}
\DeclareMathOperator{\SL}{SL}

\DeclareMathOperator{\Spec}{Spec}

\DeclareMathOperator{\Tr}{Tr}

\newcommand{\op}[1]{\operatorname{#1}}
\newcommand{\opp}{{\operatorname{op}}}
\newcommand{\mc}[1]{\mathcal{#1}}
\newcommand{\mb}[1]{\mathbb{#1}}
\newcommand{\mr}[1]{{\sf #1}}
\newcommand{\mf}[1]{\mathfrak{#1}}
\newcommand{\mch}[1]{\check{\mathcal{#1}}}
\renewcommand{\b}[1]{\bold{#1}}
\newcommand{\bs}[1]{\boldsymbol{#1}}
\newcommand{\br}[1]{\overline{#1}}

\newcommand{\dv}{\underline{\dim}}
\newcommand{\wtd}[1]{\widetilde{#1}}

\newcommand{\e}{{\sf e}}
\newcommand{\f}{{\sf f}}
\newcommand{\g}{{\sf g}}

\renewcommand{\d}{{\sf d}}
\renewcommand{\t}{{\sf t}}
\renewcommand{\S}{{\bf S}}

\newcommand{\proj}{\operatorname{proj}\text{-}}

\newcommand{\ckQ}{\widehat{k\Delta}}

\newcommand{\sm}[1]{{\left(\begin{smallmatrix}#1\end{smallmatrix}\right)}}

\newcommand{\mat}[1]{{\begin{matrix}#1\end{matrix}}}
\newcommand{\fr}[1]{\framebox[1.2\width]{{${#1}$}}}

\newcommand{\uca}{\br{\mc{C}}}

\newcommand{\kllm}{{K_{l,l}^m}}
\newcommand{\zllm}{{Z_{l,l}^m}}

\newcommand{\invint}[1]{{\scriptsize \overleftarrow{[#1]}}}
\newcommand{\invarr}[1]{{\scriptsize \overleftarrow{#1}}}
\renewcommand{\ss}[1]{{\scriptstyle{#1}}}

\newcommand{\bl}{{\beta_l}}
\newcommand{\circm}{{\mathop{\text{\textcircled{\tiny \textit m}}}\nolimits}}

\newcommand{\ijn}{_{i,j}^{n}}
\newcommand{\bigsimeq}{\wtd{=\joinrel=}}
\newcommand{\Diamondsh}{\mathrel{\Diamond\!\!\raisebox{0.45pt}{$\rangle$}}}

\begin{document}

\title{Cluster Algebras, Invariant Theory, and Kronecker Coefficients II}
\author{Jiarui Fei}
\address{Shanghai Jiao Tong University, School of Mathematical Sciences}
\email{jiarui@sjtu.edu.cn}
\thanks{}
\dedicatory{Dedicated to Professor Sergey Fomin on the Occasion of his Sixtieth Birthday}


\date{}

\begin{abstract} We prove that the semi-invariant ring of the standard representation space of the $l$-flagged $m$-arrow Kronecker quiver is an upper cluster algebra for any $l,m\in \mb{N}$.
The quiver and cluster are explicitly given.
We prove that the quiver with its rigid potential is a polyhedral cluster model.
As a consequence, to compute each Kronecker coefficient $g_{\mu,\nu}^\lambda$ with $\lambda$ at most $m$ parts,
we only need to count lattice points in at most $m!$ fibre (rational) polytopes inside the $\g$-vector cone,
which is explicitly given. 
\end{abstract}

\maketitle

\setcounter{tocdepth}{2}
\tableofcontents

\section*{Introduction}
In the previous paper \cite{Fk1}, we studied the semi-invariant rings $\SI_\bl(\kllm)$ of the {\em flagged Kronecker quivers} $\kllm$ with standard dimension vector $\bl$ below
$$\kronml{1}{2}{l}{l}{2}{1}{}$$
These semi-invariant rings are interesting because they are related to the important {\em Kronecker coefficients}.
We proved that when $m=2$, for each $l$ the semi-invariant ring is an {\em upper cluster algebra} \cite{BFZ}.
Moreover, the {\em ice quiver} of this upper cluster algebra has a {\em rigid potential} such that they form a {\em polyhedral cluster model}.
Roughly speaking, this means that the upper cluster algebra has a basis parametrized by the lattice points in the {\em $\mu$-supported $\g$-vector cone}.

In this paper, we settle the general case. Namely, 
\begin{theorem}[=Theorem \ref{T:equal}] For any $l,m\geq 2$, the semi-invariant ring $\SI_\bl(\kllm)$ is isomorphic to the graded upper cluster algebra $\uca(\wtd{\Diamond}_l^m,\wtd{\b{s}}_l^m; \wtd{\bs{\sigma}}_l^m)$.
\end{theorem}
\noindent We explicitly construct an initial seed $(\wtd{\Diamond}_l^m,\wtd{\b{s}}_l^m)$.
Roughly speaking, the ice quiver $\wtd{\Diamond}_l^m$ is built up in two steps.
We first glue $m-1$ so-called {\em diamond quivers} $\Diamond_l$ together, then we {\em extend} it through $m$ frozen vertices.
A single $\Diamond_l$ for $l=5$ is drawn in Figure \ref{f:Pair} (right).  
Moreover, all {\em cluster variables} in $\wtd{\b{s}}_l^m$ are Schofield's semi-invariants.
Each such upper cluster algebra strictly contains the corresponding cluster algebra. 

The statement about the cluster model is also true for the general case.
\begin{theorem}[=Theorem \ref{T:CM}] There is a rigid potential $\wtd{W}_l^m$ on $\wtd{\Diamond}_l^m$ such that $(\wtd{\Diamond}_l^m,\wtd{W}_l^m)$ is a polyhedral cluster model.
\end{theorem}
Actually, for the sake of exposition, we consider a quiver with potential $(\br{\Diamond}_l^m,\br{W}_l^m)$ {\em mutation}-equivalent to $(\wtd{\Diamond}_l^m,\wtd{W}_l^m)$.
Let $\mr{G}_l^m$ be the $\mu$-supported $\g$-vector cone of $(\br{\Diamond}_l^m,\br{W}_l^m)$. 
We can explicitly describe it by {\em stability conditions}, more precisely by the dimension vectors of subrepresentations of some {\em boundary} representations of $(\br{\Diamond}_l^m,\br{W}_l^m)$.
Combining with a result in the previous paper, we are able to efficiently compute Kronecker coefficients $g_{\mu,\nu}^{\lambda}$ associated to {\em all} triple of partitions $(\lambda,\mu,\nu)$ of $n$. 
By definition $g_{\mu,\nu}^\lambda$ is the tensor product multiplicity for irreducible representations of the symmetric group $\mf{S}_n$:
$$\S^\mu \otimes \S^\nu \cong \bigoplus_\lambda g_{\mu,\nu}^\lambda \S^\lambda.$$
To author's best knowledge, this is the fastest algorithm available so far.
\begin{theorem}[=Theorem \ref{T:KC}] Let $\mu,\nu$ (resp. $\lambda$) be partitions of length $\leq l$ (resp. $\leq m$). Then
	$$g_{\mu,\nu}^\lambda = \sum_{\omega\in \mf{S}_m(\lambda)} \op{sgn}(\omega) \Big|\mr{G}_l^m(\mu,\nu,\lambda^\omega)\cap \mb{Z}^{(\br{\Diamond}_l^m)_0}\Big|.$$
\end{theorem}
\noindent Here we write $\lambda^\omega$ for the vector in $\mb{Z}^m$ given by $\lambda^\omega(i)= \lambda(i)-i+\omega(i)$,
and set $\mf{S}_m(\lambda):=\left\{\omega\in \mf{S}_m \mid \lambda^\omega\in \mb{Z}_{\geq 0}^m \right\}$.
Moreover, $\mr{G}_l^m(\mu,\nu,\lambda^\omega)$ is the fibre polytope for the grading map $\br{\bs{\sigma}}_l^m : \mr{G}_l^m\to \mb{R}^{2l+m}$.

An earlier version of this paper had been finished by summer of 2016. In that version, we proved the same main results using the methods similar to those in the first paper. However, with this application in mind and motivated by the work of Gross-Hacking-Keel-Kontsevich \cite{GHKK}, we developed some machinery in \cite{FW}. The proof in the current version is much simpler and more conceptual than the previous one.  
The paper is almost self-contained -- we only assume Theorem 0.1 from the first paper.

Here is the outline of this paper.
In Section \ref{S:Bobjs}, we introduce the basic objects to study in this paper, including an important intermediate object $\zllm$.
Lemma \ref{L:embedding} is a key lemma.
In Section \ref{S:SSI}, we introduce a special class of Schofield's semi-invariants for $\Rep_{\bl}(\kllm)$, which will be our candidates for the cluster variables.
In Section \ref{S:UCA}, we recall the definition of graded upper cluster algebras.
In Section \ref{S:CSZ}, we describe the cluster structure of the coordinate algebra of $\zllm$, and make a connection with a construction of Fock-Goncharov.
In Section \ref{S:CSmain}, we prove our first main result (Theorem \ref{T:equal}) that $\SI_\bl(\kllm)$ is an upper cluster algebra.

In Section \ref{ss:QP}, we recall the ice quivers with potentials as cluster models following \cite{DWZ1,P,Fs1}. 
We also recall the main results in \cite{FW}. After some technical preparation in Section \ref{ss:museq},
we prove in Section \ref{ss:CM} our second main result (Theorem \ref{T:CM}) that $(\br{\Diamond}_l^m,\br{W}_l^m)$ is a polyhedral cluster model.
In Section \ref{ss:recap}, we recap the connection to the Kronecker coefficients made in the first paper.
As a consequence, we state our last main result (Theorem \ref{T:KC}) on computing Kronecker coefficients.  
In Section \ref{ss:Glm}, we briefly mention a description of the $\g$-vector cone $\mr{G}_l^m$ (Proposition \ref{P:Tv} and \ref{P:Glm}). 
We provide a detailed example when $l=m=3$.

\subsection*{Notations and Conventions}
Our vectors are exclusively row vectors. Unless otherwise stated, all modules are right modules.
For a quiver $Q$, we denote by $Q_0$ the set of vertices and by $Q_1$ the set of arrows.
Arrows are composed from left to right, i.e., $ab$ is the path $\cdot \xrightarrow{a}\cdot \xrightarrow{b} \cdot$.

Throughout the paper, the base field $k$ is algebraically closed of characteristic zero.
Unadorned $\Hom$ and $\otimes$ are all over the base field $k$, and the superscript $*$ is the trivial dual for vector spaces.
For any representation $M$, $\dv M$ is the dimension vector of $M$.
For direct sum of $n$ copies of $M$, we write $nM$ instead of the traditional $M^{\oplus n}$.

In this paper, we use the interval notation $[a,b]$ for the sequence $(a,a+1,\cdots,b)$,
and $\overleftarrow{[a,b]}$ for the reverse order sequence $(b,b-1,\cdots,a)$.
Moreover, $[n]$ is a shorthand for $[1,n]$.

\part{$\SI_\bl(\kllm)$ as an Upper Cluster Algebra} \label{P:I}
\section{Basic Objects} \label{S:Bobjs}

Let $X$ be an irreducible affine variety acted by a reductive algebraic group $G$.
We denote by $X/G:=\Spec k[X]^G$ the {\em categorical quotient} (in the category of varieties).
The typical situation throughout this paper is that there is a non-empty open subset $X^\circ\subset X$ where the action of $G$ is free. 
In particular, $X^\circ$ admits a {\em geometric quotient} \cite{PV}.
 
Let $Q=(Q_0,Q_1)$ be a finite quiver.
For a dimension vector $\beta$ of $Q$ (that is an element in $\mb{Z}^{Q_0}$), let $V$ be a $\beta$-dimensional vector space $\prod_{i\in Q_0} k^{\beta(i)}$. We write $V_i$ for the $i$-th component of $V$.
The space of all $\beta$-dimensional representations is
$$\Rep_\beta(Q):=\bigoplus_{a\in Q_1}\Hom(V_{t(a)},V_{h(a)}).$$
Here we denote by $t(a)$ and $h(a)$ the tail and head of $a$.
The product of general linear group
$$\GL_\beta:=\prod_{i\in Q_0}\GL(V_i)$$
acts on $\Rep_\beta(Q)$ by the natural base change.
But in this paper we are mainly interested in the action of the subgroup $\SL_\beta:=\prod_{i\in Q_0}\SL(V_i)\subset \GL_\beta$.
We define the rings of semi-invariants
$$\SI_\beta(Q):=k[\Rep_\beta(Q)]^{\SL_\beta}.$$

We call the linear quiver $A_l$ with the standard dimension vector $\bl$ a {\em complete flag} of length $l$
$$\Al{1}{2}{l}.$$  
The opposite quiver $A_l^\vee:=A_l^{\opp}$ of $A_l$ with the same dimension vector $\bl$ is called its {\em dual}
$$\Alv{1}{2}{l}.$$  
Let $\Rep_\bl^\circ(A_l)$ (resp. $\Rep_{\bl}^\circ(A_l^\vee)$) be the open subset of $\Rep_\bl(A_l)$ (resp. $\Rep_{\bl}(A_l^\vee)$) where each linear map is injective (resp. surjective).
Let $F\in \Rep_\bl^\circ(A_l)$ (resp. $\check{F}\in \Rep_{\bl}^\circ(A_l^\vee)$) be the representation given by $(I_k,0)$ (resp. $\sm{I_k \\ 0}$) on each arrow $k\to k+1$ (resp. $k+1\to k$), where $I_k$ is the $k\times k$ identity matrix. 

Throughout we set $G:=\SL_l$. Let $U^-$ (resp. $U:=U^+$) and $H$ be the subgroups of lower (resp. upper) triangular and diagonal matrices in $G$. 
For any $l\in\mb{N}$, we let $\gamma_l:=(1,2,\dots,l-1)$ so that we have the decomposition $\SL_{\bl}=\SL_{\gamma_{l}} \times \SL_l$.
We denote by $\mc{A}$ and $\mch{A}$ the categorical quotient $\Rep_\bl(A_l)/\SL_{\gamma_{l}}$ and $\Rep_\bl(A_l^\vee)/\SL_{\gamma_{l}}$ respectively.
The group $\SL_{\gamma_{l}}$ acts freely on $\Rep_\bl^\circ(A_l)$ so $\mc{A}^\circ:=\Rep_\bl^\circ(A_l)/\SL_{\gamma_{l}}$ is a geometric quotient. 
Similarly $\mch{A}^\circ:=\Rep_\bl^\circ(A_l^\vee)/\SL_{\gamma_{l}}$ is also a geometric quotient.
By abuse of notation, we still write $F$ and $\check{F}$ for the corresponding points in $\mc{A}$ and $\mch{A}$. 
Since the action of $\GL_l$ commutes with that of $\SL_{\gamma_{l}}$, the group $\GL_l$ acts on $\mc{A}$ and $\mch{A}$ as well.
As our convention, we change the left action of $\GL_l$ on $\mc{A}$ to the right action by taking the inverse.

\begin{lemma}[{\cite[Lemma 4.5]{FW}}] \label{L:stabilizer}  $F\cdot G$ is the only $G$-orbit in $\mc{A}^\circ$ and the stabilizer of $Fh$ is $U^-$ for any $h\in H$.
There is an obvious dual statement for $\check{F}$.
\end{lemma}

Let $p_i$ (resp. $\check{p}_i$) be the unique path from $i$ to $l$ in $A_l$ (resp. from $l$ to $i$ in $\check{A}_l$).
A pair of representations $(A_1,A_2)\in \Rep_{\bl}(A_l)^2$ is called {\em generic} if the matrix $\sm{A_1(p_i)\\A_2(p_j)}$ has full rank for all $i+j=l$.
Similarly, we can define the generic pair in $\Rep_{\bl}(A_l^\vee)\times \Rep_{\bl}(A_l^\vee)$.
A pair $(A,\check{A})$ in $\Rep_{\bl}(A_l)\times \Rep_{\bl}(A_l^\vee)$ is called {generic} if the matrix $A(p_i)\check{A}(\check{p}_i)$ has full rank for all $1\leq i\leq l-1$.
It is clear that a representation $A$ (resp. $\check{A}$) in a generic pair lies in $\Rep_{\bl}^\circ(A_l)$ (resp. $\Rep_{\bl}^\circ(A_l^\vee)$).
A pair in $\mc{A}^\circ\times \mc{A}^\circ$ (or $\mch{A}^\circ\times \mch{A}^\circ$, $\mc{A}^\circ\times \mch{A}^\circ$) is called {\em generic} if their corresponding representations are generic.
Let $(\mc{A}\times\mc{A})^\circ$, $(\mc{A}\times\mch{A})^\circ$, and $(\mch{A}\times\mch{A})^\circ$ be the open subset of $\mc{A}\times\mc{A}$, $\mc{A}\times\mch{A}$, and $\mch{A}\times\mch{A}$ consisting of generic pairs.

\begin{corollary}[{\cite[Corollary 4.6]{FW}}] \label{C:doubleflag} $G$ acts freely on $(\mc{A}\times\mc{A})^\circ$, $(\mc{A}\times\mch{A})^\circ$, and $(\mch{A}\times\mch{A})^\circ$. Their quotients are all isomorphic to $H$.
\end{corollary}


We can reformulate the generic condition in terms of the diagonal $\GL_l$-action on $\mc{A}\times {\mc{A}}$ and $\mc{A}\times \mch{A}$.
It is easy to see that $(A_1,A_2)$ (resp. $(A,\check{A})$) is generic if it is equivalent to $(F, Fhw_0)$ (resp. $(F, h\check{F})$) for some $h\in H$ under the $\GL_l$-action. 
Here $w_0$ is the longest element in the Weyl group $\mf{S}_l$ of $\GL_l$.

Let $\kllm$ be the flagged $m$-arrow Kronecker quiver of length $l$
$$\kronml{-1}{-2}{-l}{l}{2}{1}{a_1,\dots,a_m}.$$
We still denote by $\bl$ the standard dimension vector $\bl(i)=|i|$ for $\kllm$.
We decompose the representation space as 
$$\Rep_\bl(\kllm):=\Rep_\bl(A_l) \times \Hom(V_{-l},V_l)^m \times \Rep_\bl(A_l^\vee).$$
We call a map in $\Hom(V_{-l},V_l)$ a {\em central} map in $\Rep_\bl(\kllm)$.
Let $\Rep_{\bl}^\circ(\kllm)$ be the open subset of $\Rep_{\bl}(\kllm)$ where all $m$ central maps are invertible,
and $\Rep_{\bl}^1(\kllm)$ be the closed subset where all $m$ central maps are in $G=\SL_l$.
We have that \begin{equation} \label{eq:repcirc} \Rep_{\bl}^\circ(\kllm)/\SL_\bl=(\mc{A}\times (\GL_l)^m \times \mch{A})/(G\times G).
\end{equation}
Let $D_{\rm c}\in \SI_\bl(\kllm)$ be the product of central determinants
$$D_{\rm c}(M):=\prod_{n=1}^m \det(M(a_n))\ \text{ for $M\in\Rep_\bl(\kllm)$.}$$
So $\Rep_{\bl}^\circ(\kllm)=\Rep_{\bl}(\kllm)\setminus Z(D)$, where $Z(D_{\rm c})$ is the zero locus of $D_{\rm c}$.
We define 
$$\zllm:=(\mc{A} \times G^m \times \mch{A})/{(G\times G)}.$$

The following lemma is a trivial fact but we will frequently use it (sometimes implicitly). 
\begin{lemma} \label{L:locinv}
	Localization with respect to an invariant commutes with taking invariants.
\end{lemma}

\begin{lemma} \label{L:localdet} We have that $\Rep_{\bl}^\circ(\kllm)/\SL_\bl = \zllm\times (\mb{G}_{\rm m})^m$.
In particular, 
	$$\SI_\bl\left(\kllm\right)_D = k\left[\zllm\times (\mb{G}_{\rm m})^m\right].$$
\end{lemma}

\begin{proof} Recall the equality \eqref{eq:repcirc}. Consider the $G\times G$-action on $\GL_l=G\rtimes \mb{G}_{\rm m}$.
	Since the action is trivial on $\mb{G}_{\rm m}$, we have that
$$\Rep_\bl^\circ(\kllm)/\SL_\bl=(\mc{A} \times (G\rtimes \mb{G}_{\rm m})^m \times \mch{A})/(G\times G) = \zllm\times (\mb{G}_{\rm m})^m.$$
Equivalently we have that $k\big[\Rep_\bl^\circ(\kllm)\big]^{\SL_\bl} = k\big[\zllm\times (\mb{G}_{\rm m})^m\big]$.
By Lemma \ref{L:locinv}, we have that
	$$\SI_\bl(\kllm)_D = \left(k\big[\Rep_\bl(\kllm)\big]_D\right)^{\SL_\bl} = k\left[\zllm\times (\mb{G}_{\rm m})^m\right].$$
\end{proof}

For a positive integer $n$, we use the shorthand $g_{[n]}$ for the alternating product $g_1g_2^{-1}\cdots g_n^{(-1)^{n-1}}$,
and $g_{\invint{n}}$ for the product in reverse order $g_n^{(-1)^{n-1}} \cdots g_2^{-1}g_1$.
Let us consider the map $\wtd{\pi}_n:\mc{A}\times G^m\times \mch{A} \to \mc{A}^2\times \mch{A}^2$ given by 
\begin{align} \label{eq:pieven} (A,g_1,\dots,g_m,\check{A})&\mapsto (Ag_{[n-1]},Ag_n,g_{\invint{n}}\check{A},\check{A}), && n\in [2,m] \text{ even}\\
\label{eq:piodd} (A,g_1,\dots,g_m, \check{A})&\mapsto (Ag_{[n-1]},A,g_{\invint{n}}\check{A},g_n\check{A}), && n\in [3,m] \text{ odd}.
 \end{align}  
These maps followed by the quotient map $\mc{A}^2\times \mch{A}^2\to (\mc{A}^2 \times \mch{A}^2 )/{G}$ are constant on the $G\times G$-orbits.
By the universal property of the categorical quotient, we get a map
\begin{equation} \label{eq:pi} \pi_{n}: (\mc{A} \times G^m \times \mch{A} )/{(G\times G)}\to (\mc{A}^2 \times \mch{A}^2 )/{G}.\end{equation}
Note that each $\pi_n$ is dominant.
We denote the quotient $(\mc{A}^2 \times \mch{A}^2 )/{G}$ by $X_{2,2}$.
Let $X_{2,2}^\circ$ be the open subset of $X_{2,2}$ where the pairs $(A_1,{A}_2),(A_1,\check{A}_2),(A_2,\check{A}_1)$ and $(\check{A}_2,\check{A}_1)$ are all generic.
We also define $X_{2,2}^{\circm}$ to be the open subset of $X_{2,2}$ where the pairs 
\begin{align*}
\text{$(A_1,\check{A}_2)$ and $(A_2,\check{A}_1)$ are generic if $m$ is even;}\\
\text{$(A_1,{A}_2)$ and $(\check{A}_2,\check{A}_1)$ are generic if $m$ is odd.}
\end{align*}

Let $(\zllm)^{\circ}_\pi$ be the open subset of $\zllm$ such that
\begin{enumerate}
\item[] $(Ag_{[n]},A)$ and $(\check{A},g_{\invint{n}}\check{A})$ are generic for each even $n\in[2,m-1]$; 
\item[] $(Ag_{[n]},\check{A})$ and $(A,g_{\invint{n}}\check{A})$ are generic for each odd $n\in[1,m-1]$.
\end{enumerate}
In particular, $(\zllm)^{\circ}_\pi$ is contained in the following open subset of $\zllm$ 
$$(\zllm)^\circ:=(\mc{A}^\circ \times G^m \times \mch{A}^\circ)/{(G\times G)}.$$
Recall that we say that $g\in \SL_l$ is {\em Gauss decomposable} if $g=\br{u}hu$ for some $\br{u}\in U^-,\ h\in H$, and $u\in U$. 
Let $\wtd{\mf{S}}_l\subseteq \SL_l$ be some Tits' lift of the Weyl group $\mf{S}_l$.
We write $\wtd{w}$ for the lift of $w$ in $\wtd{\mf{S}}_l$.
\begin{lemma} \label{L:Gauss} For any $g\in\SL_l$, $(Fg,\check{F})$ (resp. $(Fg,F)$) is generic if and only if $g$ (resp. $g\wtd{w}_0$) is Gauss decomposable. 
\end{lemma}

\begin{proof} By the Bruhat decomposition we can write $g$ as $g=\br{u}h\wtd{w}u$ with $\br{u}\in U^-,\ h\in H,\ w\in \wtd{\mf{S}}_l$, and $u\in U$.
By Lemma \ref{L:stabilizer}, $(Fg,\check{F})=(Fh\wtd{w},\check{F})$.
It is generic if and only if $\wtd{w}$ is the identity element. The argument is similar for $(Fg,F)$.
\end{proof} 

Here is our key lemma. We suggest readers skipping its proof for the first-time read. 
The proof is easy to understand once readers have some pictures of Section \ref{S:CSZ} in mind. 
\begin{lemma} \label{L:embedding}  The restriction of the map 
	$$\pi:=(\pi_2,\dots,\pi_{m}): \zllm\to (X_{2,2})^{m-1}$$ 
	to the open set $(\zllm)^\circ_\pi$ is a closed embedding to $(X_{2,2}^\circ)^{m-2} \times X_{2,2}^{\circm}$.
\end{lemma}

\begin{proof} We first show that the map restricted to $(\zllm)^\circ_\pi$ is injective.
	By the definition of $\pi$, the image $\pi\big((\zllm)^\circ_\pi\big)$ is in $(X_{2,2}^\circ)^{m-2}\times X_{2,2}^{\circm}$.
	We have seen that $(\zllm)^\circ$ and $X_{2,2}^\circ$ are geometric quotients.
	A point in $(\zllm)^\circ$ can be represented by $(F,g_1,\dots,g_m,\check{F})$.
	We keep in mind that the stabilizer of $Fh$ and $h\check{F}$ are $U^-$ and $U$ (Lemma \ref{L:stabilizer}).
	By Lemma \ref{L:Gauss} $g_1$ is Gauss decomposable, so we can assume that $g_1\in H$.
	Suppose that it has the same image as $(F,g_1',\dots,g_m',\check{F})$ with $g_1'\in H$.
	In particular, they have the same image under $\pi_2: \zllm\to X_{2,2}$, that is,
	\begin{align} 
	\big(Fg_1,Fg_2,g_2^{-1}g_1\check{F},\check{F}\big)&=\big(Fg_1',Fg_2',{g_2'}^{-1}g_1'\check{F},\check{F}\big) \text{ in } X_{2,2}. \label{eq:X21}
	\end{align}
	Comparing the first and last factors in \eqref{eq:X21}, we see that there are some $\br{u}\in U^-$ and $u\in U$ such that
	$\br{u}^{-1}g_1u=g_1'$. 
	Since $g_1,g_1'\in H$, we must have that $\br{u}=u=e$ so $g_1=g_1'=h$.
	Since the stabilizer of the generic pair $(Fh,\check{F})$ is equal to $U^-\cap U=\{e\}$,
	\eqref{eq:X21} also implies that 
	$$g_2'g_2^{-1} \in U^- \text{ and } ({g_2'}^{-1}g_1')^{-1}g_2^{-1}{g_1}=h^{-1}g_2'{g_2}^{-1}h \in U.$$
	So $g_2'g_2^{-1}\in U^-\cap h^{-1}Uh=\{e\}$, and we conclude that $g_2=g_2'$.
	
	We will prove $g_i=g_i'$ for $i\in [1,m]$ by induction. We have treated the initial case.
	Let us assume that $g_i=g_i'$ for $i\in [1,k-1]$ and $k$ is odd, then we consider the map 
	$\pi_{k}: \zllm\to X_{2,2}$.
	We have that \begin{align}  \label{eq:X22} 
	\big(Fg_{[k-1]},F,g_{\invint{k}}\check{F},g_k\check{F}\big)&=\big(Fg_{[k-1]}',F,g_{\invint{k}}'\check{F},g_k'\check{F}\big) \text{ in } X_{2,2}.
	\end{align}
	By the induction hypothesis, we have that $g_{[k-1]}=g_{[k-1]}'$.
	Since the stabilizer of the generic pair $(Fg_{[k-1]},F)$ is $\{e\}$,
	\eqref{eq:X22} implies that 
	$$g_k^{-1}g_k' \in U \text{ and } g_{\invint{k}}^{-1}g_{\invint{k}}'=\big(g_kg_{\invint{k-1}}\big)^{-1}g_k'g_{\invint{k-1}}'\in U.$$
	By the generic assumption, $g_{\invint{k-1}} \wtd{w}_0$ is Gauss decomposable.
	For the similar reason as before, we can assume that $g_{\invint{k-1}}\wtd{w}_0=h\in H$.
	So $\wtd{w}_0h^{-1} {g_k}^{-1}g_k' h\wtd{w}_0^{-1} \in U$, and thus $g_k^{-1}g_k'\in U\cap h\wtd{w}_0^{-1}U\wtd{w}_0h^{-1}=U\cap U^-=\{e\}$.
	We conclude that $g_k=g_k'$.
	Similar argument can also treat the case when $k$ is even.
	Therefore we finished the proof by induction.

We remain to show that the image $\pi\big((\zllm)^\circ_\pi\big)$ is closed in $(X_{2,2}^\circ)^{m-2} \times X_{2,2}^{\circm}$.
Let $p_{ij}$ be the natural projection from $X_{2,2}$ to $\mc{A}^2/G$ (or $(\mc{A}\times \mch{A})/G$) induced by 
mapping $\left(A_1,A_1,\check{A}_1,\check{A}_2\right)\in \mc{A}^2\times \mch{A}^2$ to its $i,j$-th factors.
We shall write $p_{i,j}^{n}$ for the projection $p_{i,j}$ from the $(n-1)$-th copy in $(X_{2,2}^\circ)^{m-1}$ (note that the $(n-1)$-th copy is the target of the map $\pi_n$).
We recall from Corollary \ref{C:doubleflag} that the quotient $(\mc{A}\times \mc{A})^\circ/G$ or $(\mc{A}\times \mch{A})^\circ/G$ is geometric and is isomorphic to the maximal torus $H\subset G$.
We define the map $p: (X_{2,2}^\circ)^{m-2} \times X_{2,2}^{\circm} \to H^{2(2m-3)}$ by \footnote{\hbox{The components of $p$ correspond to the disk diagonals in the triangulation $\mb{T}_m$ (see Section \ref{S:CSZ})}}
$$\left( p_{1,4}^{2},p_{2,3}^{2},p_{1,2}^{2},p_{1,2}^{3},p_{3,4}^{2},p_{3,4}^{3},p_{1,4}^{3},p_{1,4}^{4},p_{2,3}^{3},p_{2,3}^{4},\dots, p_{i,j}^{m-1},p_{i,j}^{m},p_{k,l}^{m-1},p_{k,l}^{m} \right),$$
where $(i,j,k,l)=(1,2,3,4)$ if $m$ is odd; $(i,j,k,l)=(1,4,2,3)$ if $m$ is even.
It is easy to see from the definition of $\pi$ that the image $\pi\big((\zllm)^\circ_\pi\big)$ is equal to the inverse image $p^{-1}\big((\Delta_H)^{2m-3}\big)$,
where $\Delta_H$ is the diagonal in $H\times H$.
This implies the closeness of $\pi\big((\zllm)^\circ_\pi\big)$.	
\end{proof}

\section{Schofield's Semi-invariants} \label{S:SSI}

\subsection{The semi-invariants $s\ijn$}
We follow the setting in \cite{SV}.
Let $Q$ be a finite quiver and add($Q$) be the additive $k$-category generated by $Q$.
For each vertex $i$ we have an object $P_i$ in add($Q$) such that $\Hom_{\rm{add}(Q)}(P_i,P_j)$ is the vector space with basis the paths from $j$ to $i$ including the trivial path $e_i$ if $i=j$.
From now on, we shall write $\Hom_{Q}(-,-)$ for $\Hom_{\rm{add}(Q)}(-,-)$.
For any $\f\in \mb{Z}_{\geq 0}^{Q_0}$, we denote
$$P(\f):=\bigoplus_{i\in Q_0} \f(i)P_i.$$
We call an element $f$ in $\Hom_{Q}(P(\f_1),P(\f_0))$ a {\em presentation}.
The {\em weight} of $f$ is by definition $\f_1-\f_0$.
From now on, we will view a weight as an element in the dual $\Hom_{\mb{Z}}(\mb{Z}^{Q_0},\mb{Z})$ via the usual dot product.

Recall that for any presentation $f$ of weight $\f$ with $\f(\beta)=0$,
we can associate a Schofield's semi-invariant function (\cite{S1}) $s(f)$ defined by 
$s(f)(M):=\det M(f)$, where $M(f)$ is the matrix obtained from the transpose matrix of $f$ by substituting each path $p$ by $M(p)$. 
It is easy to check that $s(f)$ lies in the weight space 
$$\SI_\beta(Q)_\f:=\big\{f\in k[\Rep_\beta(Q)]\mid g(f)=\f(g)f,\ \forall g\in\GL_\beta \big\},$$
where by abuse of notation $\f(g)$ is the character map 
$$\big(g(i)\big)_{i\in Q_0}\mapsto\prod_{i\in Q_0} \big(\det g(i)\big)^{\f(i)}.$$
In this context $\f$ is also called the {\em $\sigma$-weight} of $s(f)$.

\begin{definition} A Schofield's semi-invariant is called {\em pure} if all the entries in the matrix representing $f$ are given by a single path (rather than a linear combination of paths).
\end{definition}

Let $\br{K}_{l,l}^{m}$ be the quiver 
$$\kronllmm{-1}{-2}{-l}{l}{2}{1}{}{}$$
This quiver can be obtained from $\kllm$ by adjoining $m$ arrows from $l$ to $-l$.
We label the $m$ new arrows by $a_n^{-1}$ for $n\in[1,m]$.
We impose the relations
$$a_n a_n^{-1} = e_{-l}\ \text{ and }\ a_n^{-1} a_n = e_{l}\ \text{ for each $n$}.$$ 
So we can view each $a_n^{-1}$ as the formal (two-sided) inverse of $a_n$.
A representation $M$ of $\br{K}_{l,l}^{m}$ with the above relation satisfies that 
$$M(a_n)M(a_n^{-1})\ \text{ and }\ M(a_n^{-1})M(a_n)\ \text{ are identity matrices.}$$
This in particular implies that each $M(a_n)\in \GL_l$.
Throughout this paper, a representation of $\br{K}_{l,l}^{m}$ always satisfies these relations.

For any finite sequence $I:=(i_1,i_2,\dots,i_n)$ of numbers in $[1,m]$, we can associate a path from $-l$ to $\pm l$ ($\pm$ depends on the parity of the length of $I$)
$$\bs{a}_I:= a_{i_1} a_{i_2}^{-1} a_{i_3} \cdots a_{i_n}^{(-1)^{n-1}}.$$  
Now each triple $(i,j,I)$ with $n$ odd determines a unique path from $-i$ to $j$ via $I$:
$$p_{-i,j}^I:=p_{-i,-l} \bs{a}_I p_{l,j},$$
where $p_{-i,-l}$ (resp. $p_{l,j}$) is the unique path from $-i$ to $-l$ (resp. from $l$ to $j$).
In this paper, $I$ is always a subinterval in $[1,m]$ with usual or reverse order. 

For $I$ and $J$ both of odd length, let $f_{i,j;k}^{I,J}$ be the presentation
\begin{equation} \label{eq:presIJ} P_k \xrightarrow{\left(p_{-i,k}^I,p_{-j,k}^J\right)} P_{-i}\oplus P_{-j}.
\end{equation}
In this paper, we frequently abbreviate $I$ for $p_{-i,k}^I$ as a map in $\Hom_{\br{K}_{l,l}^{m}}(P_{k},P_{-i})$.
So we shall write $P_k \xrightarrow{(I,J)} P_{-i}\oplus P_{-j}$ for \eqref{eq:presIJ}.
When $i+j=k$, we will write $f_{i,j}^{I,J}$ for $f_{i,j;k}^{I,J}$. 
So $s\big(f_{i,j}^{I,J}\big)$ is a pure Schofield semi-invariant of weight $\e_k-\e_{-i}-\e_{-j}$.
By \cite[Lemma 1.8]{Fs1}, all $s\big(f_{i,j}^{I,J}\big)$ are irreducible.
Another convention is that if $i$ or $j$ is zero, then we set $P_{-i}$ or $P_{-j}$ to be zero.

We extend the involution $\vee$ on the path algebra $k\kllm$ defined in \cite{Fk1} \footnote{It is denoted by $-$ there.} to the $k$-category generated by $\br{K}_{l,l}^{m}$ as follows.
It sends the path $p_{-i,j}^I$ to $p_{-j,i}^{\invarr{I}}$, and thus induces a map
\begin{equation*} \label{eq:-map}
\vee: \Hom_{\br{K}_{l,l}^{m}}\Big(\bigoplus_j P_j, \bigoplus_i P_{-i}\Big) \xrightarrow{} \Hom_{\br{K}_{l,l}^{m}}\Big(\bigoplus_i P_{i}, \bigoplus_j P_{-j}\Big).
\end{equation*}
We will denote $\vee f_{i,j}^{I,J}$ by $\check{f}_{i,j}^{I,J}$:
$P_{i}\oplus P_{j} \xrightarrow{\big(\invarr{I},\invarr{J}\big)} P_{-k}.$


In most part of this paper, we only care about the cases when $I=n$ and $J=[n]\text{ or }[n-1]$ depending on the parity of $n$. 
In fact, they are\begin{align*}	
f\ijn:=f_{i,j}^{n,[n-1]}: P_{i+j}&\xrightarrow{\left(n,[n-1]\right)} P_{-i}\oplus P_{-j} && \text{if $n$ is even;}\\
f\ijn:=f_{i,j}^{n,[n]}: P_{i+j}&\xrightarrow{\left(n,[n]\right)} P_{-i}\oplus P_{-j} && \text{if $n$ is odd,}
\end{align*}
and their duals $\check{f}_{i,j}^{n}$. We define $s\ijn:=s(f\ijn)$ and $\check{s}\ijn:=s(\check{f}\ijn)$.


\subsection{Lifting $s\ijn$} \label{ss:liftsi}
In this subsection we lift each $s\ijn$ to a function $\wtd{s}\ijn\in\SI_\beta(\kllm)$.
\begin{definition} \label{D:liftsijn} For any $n\in [1,m]$, we define 
	$$\wtd{s}\ijn(M):= \begin{cases}
		s\ijn & \text{if } j=0;\\
		s\ijn(M)\det\left(M\big(a_{\rm ev}^{[n]}\big)\right) & \text{if } 1\leq i+j<l \text{ or $n$ is even};\\
		\wtd{s}_{i,j}^{n-1}(M) & \text{if } i+j=l \text{ and $n$ is odd}.
\end{cases}$$
where we set $\det\left(M\big(a_{\rm ev}^{[n]}\big)\right):=\prod_{i\in[n-1]\rm\ even} \det M(a_i)$.
Similarly we define $\wtd{\check{s}}\ijn$ by replacing $s\ijn$ with $\check{s}\ijn$ in the above formula. 
\end{definition}

\begin{remark} One reason why these cases should be treated separately is that otherwise Corollary \ref{C:irreducible} below will not hold.
	More reasons can be found in Section \ref{ss:liftseed}.
\end{remark}

We are going to show that $\wtd{s}\ijn$ is regular on $\Rep_{\bl}(\kllm)$. 
Suppose $A,B,C,D$ are $p\times p,\ p\times q,\ q\times p,\ q\times q$ matrices respectively.
Consider the block matrix $M=\sm{A & B \\ C & D}$ with $D$ nonsingular.
The {\em Schur complement} $M/D$ of the block $D$ in $M$ is
$A-BD^{-1}C$.
Similarly if $A$ is nonsingular, the Schur complement $M/A$ is $D-CA^{-1}B$.
Here is a standard fact in linear algebra.

\begin{lemma} \label{L:blockdet} $\det(M)=\det(D)\det(M/D)$.
\end{lemma}

Let $\wtd{f}\ijn$ be the following presentation 
	\begin{align*} 	& P_{i+j}\oplus rP_l \xrightarrow{\left[\begin{array}{cccccc}
		n& 0 & 1 & 0  & \cdots & 0\\
		0& 0 & 2 & 3 & \ddots & \vdots\\
		\vdots&\vdots & \ddots &\ddots &\ddots &0\\
		0&0&\cdots& 0 & n-4&n-3 \\
		0&n-1&0&\cdots &0&n-2
		\end{array}\right]} P_{-i}\oplus P_{-j}\oplus rP_{-l} & & \text{for $n=2r+2$},\\
	& P_{i+j}\oplus rP_l \xrightarrow{\left[\begin{array}{cccccc}
		n& 0 & 1 & 0  & \cdots & 0\\
		0& 0 & 2 & 3 & \ddots & \vdots\\
		\vdots&\vdots & \ddots &\ddots &\ddots &0\\
		0&0&\cdots& 0 & n-3&n-2 \\
		0&n&0&\cdots &0&n-1
		\end{array}\right]} P_{-i}\oplus P_{-j}\oplus rP_{-l} & & \text{for $n=2r+1$}.
	\end{align*}

\noindent We remind readers that according to our convention each entry $n$ in the matrix represents the unique path passing $a_n$.

\begin{lemma} \label{L:wtds} Suppose that $j\neq 0$, then we have that
	$$\wtd{s}\ijn=\begin{cases} (-1)^{jr}s(\wtd{f}\ijn) & \text{if } 1\leq i+j<l \text{ or $n$ is even} ;\\
	(-1)^{j(r-1)}s(\wtd{f}_{i,j}^{n-1})& \text{if } $i+j=l$ \text{ and $n$ is odd}.\end{cases}$$
In particular, $\wtd{s}\ijn\in \SI_\bl(\kllm)$. 	
\end{lemma}

\begin{proof} We only prove the statement for even $n$. The case for odd $n$ is similar.
	We decompose $\wtd{f}\ijn$ in blocks $\wtd{f}\ijn=\sm{f_{A} & f_{B} \\ f_{C} & f_{D}}$, where
	\begin{align*}f_{A}&: P_{i+j}\xrightarrow{(n,0)} P_{-i}\oplus P_{-j},\\
	f_{B}&: P_{i+j}\xrightarrow{(1,0,\cdots,0)} rP_{-l},\\
	f_{C}&: rP_l\xrightarrow{\sm{0&\cdots&0&0\\0&\cdots&0&n-1}^T}P_{-i}\oplus P_{-j},\\
	f_{D}&: rP_l \xrightarrow{} rP_{-l}.
	\end{align*}
	The Schur complement of $M(f_{D})$ in $M(f\ijn)$ is $M(f_{A})-M(f_{C})M(f_{D})^{-1}M(f_{B})$.
	Elementary linear algebra shows that $M(f_{D})^{-1}=M(f_{D}^{-1})$, where
	$$f_{D}^{-1}: r P_{l} \xrightarrow{
		\sm{a_2^{-1} & -a_2^{-1}a_3a_4^{-1} & a_2^{-1}a_3a_4^{-1}a_5a_6^{-1} &\cdots & z\\
			0 & a_4^{-1} & -a_4^{-1}a_5a_6^{-1} &\ddots & \vdots\\
			\vdots &\ddots &\ddots &\ddots& \vdots\\
			\vdots & \ddots & \ddots & \ddots& -a_{n-4}^{-1}a_{n-3}a_{n-2}^{-1}\\
			0&\cdots&\cdots &0&a_{n-2}^{-1} }} r P_{-l}.$$
	In particular, $z=(-1)^{r-1}\prod_{i=2}^{n-2} a_i^{(-1)^{i-1}}=(-1)^{r-1}a_1^{-1}a_{[n-2]}.$
	Then $$f_{A}-f_{B}f_{D}^{-1}f_{C}=\big(n,(-1)^r [n-1]\big): P_{i+j}\xrightarrow{} P_{-i}\oplus P_{-j},$$
	so $s\ijn=(-1)^{jr}\det\left( M(f_A-f_{B}f_{D}^{-1}f_{C}) \right)$.
	Note that 
	$$M(f_{A})-M(f_{C})M(f_{D})^{-1}M(f_{B})=M(f_A-f_{B}f_{D}^{-1}f_{C}).$$
	By Lemma \ref{L:blockdet},
	\begin{align*}s(\wtd{f}\ijn)(M) &= \det\left( M(f_A-f_{B}f_{D}^{-1}f_{C}) \right)\det(M(f_{D}))\\
	&=(-1)^{jr}s\ijn(M)\prod_{i\in[n]\rm\ even} \det M(a_i)\\
	&=(-1)^{jr}\wtd{s}\ijn(M).\end{align*}
\end{proof}

Let $W$ be the $m$-dimensional vector space spanned by the arrows from $-l$ to $l$.
Let $T$ be the subgroup of diagonal matrices in $\GL(W)$.
The semi-invariant ring $\SI_\beta(K_{l_1,l_2}^m)$ is also graded by the $T$-weights $\lambda$.
We call $\wtd{\sigma}:=(\sigma,\lambda)$ the extended $\sigma$-weight. 
Note that $\sigma$-weight of $\wtd{s}\ijn$ is equal to 
\begin{equation*} \label{l:sigmaijn} \sigma\ijn:=\begin{cases}
\e_{i+j}-\e_{-i}-\e_{-j}& j=0;\\
\e_{i+j}-\e_{-i}-\e_{-j}+(r-1)(e_{l}-\e_{-l})& i+j=l \text{ and $n=2r+1$};\\
\e_{i+j}-\e_{-i}-\e_{-j}+r(e_{l}-\e_{-l})& \text{otherwise }. 
\end{cases} \end{equation*}

We set $\epsilon(k):=(-1)^{k-1}$. We write $[a]_+$ for $\max(a,0)$. 
\begin{lemma} \label{L:Tdegree} Each $\wtd{s}\ijn$ is $T$-homogeneous of degree equal to
	$$\lambda\ijn:= \begin{cases} 
	i\e_n & j=0;\\ 
	i\e_n+ \sum_{k\in [n]} \left([\epsilon(k)]_+j\e_k + [-\epsilon(k)]_+(l-j)\e_k \right) & j\geq 1,\ i+j<l,\ n \text{ odd};\\
	i\e_n+ \sum_{k\in [n-1]} \left([\epsilon(k)]_+j\e_k + [-\epsilon(k)]_+(l-j)\e_k \right) & j\geq 1,\ n \text{ even, or } i+j=l,\ n \text{ odd}.
	\end{cases}$$
\end{lemma}

\begin{proof} The case when $j=0$ is clear. For the rest,
	we observe that $\det a_k$ is $T$-homogeneous of degree $l\e_k$ and ${s}\ijn$ is $T$-homogeneous of degree
$$\begin{cases}
i\e_n + j\sum_{k\in [n-1]} \epsilon(k)\e_k &  \text{if $n$ is even};\\
i\e_n + j\sum_{k\in [n]} \epsilon(k)\e_k & \text{if $n$ is odd and $i+j< l$}.
\end{cases}$$
Then the result follows easily from Lemma \ref{L:wtds}.
\end{proof}

\begin{corollary} \label{C:irreducible} Each $\wtd{s}\ijn$ is irreducible in $\SI_\bl(\kllm)$.
\end{corollary}

\begin{proof} By \cite[Lemma 1.8]{Fs1}, it suffices to show that each $(\sigma\ijn,\lambda\ijn)$ is an extremal $\wtd{\sigma}$-weight.
Any decomposition of $\sigma\ijn$ into $\sigma$-weights in $\SI_\bl(\kllm)$ must contain some multiple of $e_{l}-\e_{-l}$, say $p(e_{l}-\e_{-l})$.
According to \cite{DW1,SV,DZ} the semi-invariant of this $\sigma$-weight is spanned by the Schofield's semi-invariants $s(f)$, where $f: p P_{l}\to p P_{-l}$.
So if it is $T$-homogeneous, then its $\lambda$-weight must be equal to $l\mu$ for some non-zero $\mu\in \mb{Z}_{\geq 0}^m$.
Then $\lambda\ijn - l\mu$ is not a $\lambda$-weight because it has some negative coordinate. A contradiction.
\end{proof}
\noindent In view of this lemma, each $\wtd{s}\ijn$ minimally lift ${s}\ijn$ to $\SI_\bl(\kllm)$.

We leave it to readers to formulate the analogous statement for the duals $\wtd{\check{s}}\ijn$.

\section{Graded Upper Cluster Algebras} \label{S:UCA}
The upper cluster algebras needed in this paper are so-called skew-symmetric and of geometric type. 
To define such an upper cluster algebra one needs to specify a {\em seed} $(\Delta,\b{x})$ in some ambient field $\mc{F}\supset k$.
Here, $\Delta$ is an ice quiver with no loops or oriented 2-cycles and 
the {\em extended cluster} $\b{x}=\{x_1,x_2,\dots,x_q\}$ is a collection of algebraically independent (over $k$)
elements of $\mc{F}$ attaching to each vertex of $\Delta$.

An {\em ice quiver} $\Delta=(\Delta_0,\Delta_1)$ is a quiver, where some vertices in $\Delta_0$ are designated as {\em mutable} while the rest are {\em frozen}.
We denote the set of mutable (resp. frozen) vertices of $\Delta$ by $\Delta_\mu$ (resp. $\Delta_\nu$).
We usually label the quiver such that the first $p$ vertices are mutable.
Up to arrows between frozen vertices, such a quiver is uniquely determined by its {\em $B$-matrix} $B(\Delta)$.
It is a $p\times q$ matrix given by
$$b_{u,v} = |\text{arrows }u\to v| - |\text{arrows }v \to u|.$$
The elements of $\b{x}$ associated with the mutable vertices are called {\em cluster variables}; they form a {\em cluster}.
The elements associated with the frozen vertices are called {\em frozen variables}, or {\em coefficient variables}.

\begin{definition} \label{D:Qmu}
	Let $u$ be a mutable vertex of $\Delta$.
	The {\em quiver mutation} $\mu_u$ transforms $\Delta$ into the new quiver $\Delta'=\mu_u(\Delta)$ via a sequence of three steps.
	\begin{enumerate}
		\item For each pair of arrows $v\to u\to w$, introduce a new arrow $v\to w$; 
		\item Reverse the direction of all arrows incident to $u$;
		\item Remove all oriented 2-cycles.
	\end{enumerate}
\end{definition}

\begin{definition} 
	A {\em seed mutation} $\mu_u$ at a (mutable) vertex $u$ transforms $(\Delta,\b{x})$ into the seed $(\Delta',\b{x}')=\mu_u(\Delta,\b{x})$ defined as follows.
	The new quiver is $\Delta'=\mu_u(\Delta)$.
	The new extended cluster is
	$\b{x}'=\b{x}\cup\{x_{u}'\}\setminus\{x_u\}$
	where the new cluster variable $x_u'$ replacing $x_u$ is determined by the {\em exchange relation}
	\begin{equation*} \label{eq:exrel}
	x_u\,x_u' = \prod_{v\rightarrow u} x_v + \prod_{u\rightarrow w} x_w.
	\end{equation*}
\end{definition}

\noindent We note that the mutated seed $(\Delta',\b{x}')$ contains the same
coefficient variables as the original seed $(\Delta,\b{x})$.
It is easy to check that one can recover $(\Delta,\b{x})$
from $(\Delta',\b{x}')$ by performing a seed mutation again at $u$.
Two seeds $(\Delta,\b{x})$ and $(\Delta^\dag,\b{x}^\dag)$ that can be obtained from each other by a sequence of mutations are called {\em mutation-equivalent}, denoted by $(\Delta,\b{x})\sim (\Delta^\dag,\b{x}^\dag)$.

Let $\mc{L}(\b{x})$ be the Laurent polynomial algebra in $\b{x}$ over the base field $k$.
If $\Delta$ is an ice quiver, we denote by $\mc{L}_{\Delta}(\b{x})$ the Laurent polynomial in $\b{x}$ which is polynomial in $\b{x}(\Delta_{\nu})$,
that is, $\mc{L}_\Delta(\b{x}):=k\left[\b{x}(\Delta_\mu)^{\pm1},\b{x}(\Delta_\nu)\right]=k\left[x_1^{\pm 1},\dots,x_p^{\pm 1}, x_{p+1}, \dots x_{q}\right]$.


\begin{definition}[{Upper Cluster Algebra}]
	The {\em upper cluster algebra} (or UCA in short) with seed $(\Delta,\b{x})$ is
	$$\uca(\Delta,\b{x}):=\bigcap_{(\Delta^\dag,\b{x}^\dag) \sim (\Delta,\b{x})}\mc{L}_{\Delta}(\b{x}^\dag).$$
\end{definition}
\noindent Note that our definition of UCA is slightly different from the original one in \cite{BFZ}, where $\mc{L}_\Delta(\b{x}^\dag)$ is replaced by $\mc{L}(\b{x}^\dag)$.
The Laurent Phenomenon \cite{FZ1,BFZ} says that a UCA contains all cluster and coefficient variables.

Any UCA, being a subring of a field, is an integral domain (and under our conventions, a $k$-algebra).
Since normality is preserved under localization and intersection, any UCA is normal.
The next lemma is useful to identify a UCA as a subalgebra of some given Noetherian normal domain.

Let $R$ be a finitely generated $k$-algebra.
We call two elements of $R$ {\em coprime in codimension 1} if the locus of their common zeros
has codimension~$\ge 2$ in $\operatorname{Spec}(R)$. 
\begin{definition}  \label{D:CR1} We say that a seed $(\Delta,\b{x})$ is {\em CR1} in $R$ if 
	\begin{enumerate} \item $\b{x}\subset R$ and each $x_u'\in R$.
		\item	each pair of cluster variables in $\b{x}$ and each pair $(x_u,x_u')$ are coprime in codimension 1 in $R$.
	\end{enumerate}
\end{definition}

\begin{lemma}[{\cite[Proposition 3.6]{FP}}] \label{L:RCA}
	Let $R$ be a finitely generated $k$-algebra and a normal domain.
	If $(\Delta,\b{x})$ is a CR1 seed in $R$, then $R\supseteq\uca(\Delta,\b{x})$.
\end{lemma}

Let $\bs{e}$ be a subset of $\Delta_0$. We write $\b{x}(\bs{e})$ for the set $\{x_e\}_{e\in\bs{e}}$.
If $\b{x}(\bs{e})$ is contained in a ring $R$, then we write $R_{\b{x}(\bs{e})}$ for the localization of $R$ at the product $\prod_{e\in\bs{e}}x_e$.
Let $\Delta^{\bs{e}}$ be the ice quiver obtained from $\Delta$ by freezing every vertex in $\bs{e}$.
\begin{lemma}[{\cite[Lemma 2.4]{Fs2}}] \label{L:reduce2frozen} Let $\bs{e}$ be a subset of $\Delta_\mu$.
	Suppose that $(\Delta,\b{x})$ is a CR1 seed in $\uca(\Delta,\b{x})$, which is Noetherian.
	Then $\uca(\Delta,\b{x})_{{\b{x}}(\bs{e})}=\uca(\Delta^{\bs{e}},\b{x})_{{\b{x}}(\bs{e})}$.
\end{lemma}

\begin{definition}[\cite{Fs1}] \label{D:wtconfig} A {\em weight configuration} $\bs{\sigma}$ of $\mb{Z}^n$ on an ice quiver $\Delta$ is an assignment for each vertex $v$ of $\Delta$ a (weight) vector $\bs{\sigma}(v)\in \mb{Z}^n$ such that for each mutable vertex $u$, we have that
	\begin{equation} \label{eq:weightconfig}
	\sum_{v\to u} \bs{\sigma}(v) = \sum_{u\to w} \bs{\sigma}(w).
	\end{equation}
\end{definition}

\noindent 
By slight abuse of notation, we can view $\bs{\sigma}$ as a matrix whose $v$-th row is the weight vector $\bs{\sigma}(v)$.
In this matrix notation, the condition \eqref{eq:weightconfig} is equivalent to $B\bs{\sigma}$ is a zero matrix.

Given a weight configuration $\bs{\sigma}$ on $\Delta$,
we can assign a multidegree (or weight) to the UCA $\uca(\Delta,\b{x})$ by setting
$\deg(x_v)=\bs{\sigma}(v)$ for $v\in\Delta_0$.
Then mutation preserves multihomogeneousity.
We say that this UCA is $\bs{\sigma}$-graded, and denoted by $\uca(\Delta,\b{x};\bs{\sigma})$.

\section{$k[\zllm]$ as a UCA} \label{S:CSZ}

\subsection{Cluster structure of $k[X_{2,2}]$} \label{ss:X22}
We denote the quotient $(\mc{A}^p \times \mch{A}^q )/{G}$ by $X_{p,q}$.
The coordinate algebra $k[X_{p,q}]$ is the same as semi-invariant ring of the quiver representations of $S_l^{p,q}$ with the standard dimension vector $\bl$:
$$\pqflag{l}{p}{q}$$

It was proved in \cite{Fs1} that the algebra $k[X_{3,0}]$ is an upper cluster algebra.
The quotient $X_{3,0}$ is isomorphic to $X_{2,1}$ and $X_{1,2}$, which are more useful for our purpose.
We shall describe the seed $(\Delta_l,\b{s}_l)$ for $k[X_{2,1}]$ directly.
To do this, we only need to specify an oriented triangle with vertices labeled by (dual) flags:
$$\mtwoup{\mch{A}_1}{\mc{A}_1}{\mc{A}_2}{$\circlearrowleft$}{}{}{}{}{}$$
We will denote such an oriented triangle by a cyclically ordered triples, eg., $$[\mch{A}_1,\mc{A}_1,\mc{A}_2]=[\mc{A}_1,\mc{A}_2,\mch{A}_1]=[\mc{A}_2,\mch{A}_1,\mc{A}_1].$$
The ice quiver $\Delta_l$ is the (ice) {\em hive quiver} of size $l$.
It has a cyclic symmetric. However, we label its vertices by two numbers in a non-symmetric manner according to the vertices in the triangle.
The labels of the frozen vertices on the side $\mch{A}_1\mc{A}_i$ ($i=1,2$) have zeros in their $j$-th coordinates ($i\neq j$).
The {\em orientation} of $\Delta_l$ is determined by the orientation of the triangle.
Here is a picture of hive quiver $\Delta_5$ for the above oriented triangle.
The dashed lines are called {\em edges} of $\Delta_l$. 
\begin{figure}[!h]
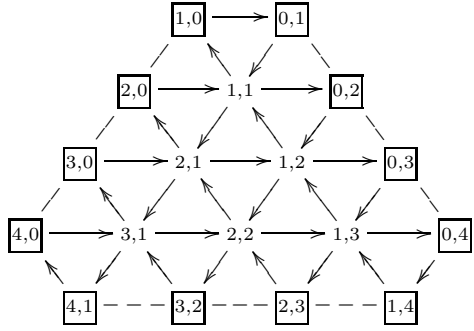
 $\hivefive$ \caption{The hive quiver $\Delta_5$} \label{f:hive} \end{figure}
The cluster variable $\b{s}_l(i,j)=s_{i,j}$ is the semi-invariant associated to the presentation
$$f_{i,j}: P_{i+j} \xrightarrow{\left(p_{i}, p_{j}\right)} P_{i}^1\oplus P_{j}^2,$$
where $p_{i}$ is the unique path in $\Hom_{S_l^{2,1}}(P_{i+j},P_{i}^1)$.
To define the cluster structure on $k[X_{1,2}]$ we only need to dualize the construction for $k[X_{2,1}]$.
Its cluster variable $\check{\b{s}}_l(i,j)=\check{s}_{i,j}$ is the semi-invariant associated to the projective presentation
$$\check{f}_{i,j}: P_{i}^1\oplus P_{j}^2 \xrightarrow{\left(p_{i}, p_{j}\right)} P_{i+j}.$$
We can reformulate the generic condition in terms of these $s_{i,j}$ on the edges.

\begin{lemma} \label{L:generic-by-edge} A pair $(A_1,\check{A}_1)\in \mc{A}_1\times \mch{A}_1$ (resp. $(A_1,A_2)\in \mc{A}_1\times \mc{A}_2$) is generic if and only if it is not vanishing on any $s_{i,0}$ (resp. $s_{i,j}\ (i+j=l)$).
\end{lemma}

Given a pair of oriented triangles, we can glue them along a pair of edges. For example,
\begin{figure}[!h] $\hspace{-0.1in} \mthree{}{}{}{}{$\circlearrowleft$}{$\circlearrowleft$}{}{} \hspace{1.8in} \mthree{}{}{}{}{$\circlearrowleft$}{$\circlearrowright$}{}{}$ \caption{} \label{f:pair} \end{figure} 

\noindent They correspond to the gluing of two ice hive quivers of the same size by identifying frozen vertices along the common edge.
After the identification, we unfreeze those frozen vertices, and add additional arrows depending on the orientations as illustrated in Figure \ref{f:Pair}.
The additional arrows are in red. The gluing on the left (resp. right) is called {\em consistent} (resp. {\em non-consistent}).
\begin{figure}[!h] $\diamondfive{}{} \qquad  \diamondfiveanti$ \caption{} \label{f:Pair} \end{figure}

More generally, let us recall that an {\em ideal oriented triangulation} of a marked bordered surface $\mc{S}$ is an oriented triangulation of $\mc{S}$ whose
vertices are the marked points of $\mc{S}$. All triangulations in this paper will be ideal and oriented.
Let $\mb{T}$ be a triangulation of a marked bordered surface $\mc{S}$.
Then we can glue ice hive quivers of size $l$ according to such a triangulation.
We denote the resulting ice quiver by $\Diamond_l(\mb{T})$.
We allow self-fold triangles and two triangles glue along more than one edge. 
Note that they happen only when there are {\em punctures} (=marked points on the interior of $\mc{S}$).
Figure \ref{f:twoglue} is a picture for such a triangulation.
\begin{figure}[!h] $\hspace{.6in} \mtwoglue{}{\sbt}{}{$\circlearrowright$}{$\circlearrowleft$} \hspace{1.2in} \ckrontwofivefive$ \caption{} \label{f:twoglue} \end{figure} 

The result for $k[X_{3,0}]$ has been generalized to $k[X_{p,0}]$ as an example in \cite{FW}.
Since $k[X_{p,q}]$ is isomorphic to $k[X_{p+q,0}]$, we actually get a cluster structure for all $k[X_{p,q}]$.
For the rest of this subsection, we shall give a seed for $k[X_{2,2}]$.
 
We knew from \cite{FW} that the ice quiver of the upper cluster algebra $k[X_{4,0}]$ corresponds to the triangulation of a rectangle as in Figure \ref{f:pair} (left).
However, we will consider a seed of $k[X_{2,2}]$ different from the one directly transferring from that of $k[X_{4,0}]$.
For this, let us recall an operation for $\mb{T}$ called twist defined in \cite{FW}.
To describe this operation, we need to specify a triangle $\t\in\mb{T}$ and one of its edges $\e$.
The twist consists of 3 steps. \begin{enumerate}
	\item Cut along the other two edges of the chosen triangle (see Figure \ref{f:twist} left); 
	\item Change the orientation of the triangle and the identification of the two edges 
	(the new identification is indicated by arrows, see Figure \ref{f:twist} right);
	\item Glue according to the new identification.
\end{enumerate}
\begin{figure}[!h]
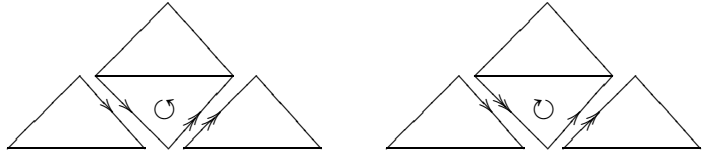
 $\mtwod{>}{}{} \hspace{-.3in} \mthree{}{}{}{}{}{$\circlearrowleft$}{>}{>>}\hspace{-.3in} \mtwod{}{>>}{}
	\hspace{.3in} \mtwod{>}{}{} \hspace{-.3in} \mthree{}{}{}{}{}{$\circlearrowright$}{>>}{>}\hspace{-.3in} \mtwod{}{>>}{}$ \caption{The twist} \label{f:twist}\end{figure}
We have a few remarks. In Step (1), if any of two edges is a part of boundary, then we do nothing for that edge.
Readers should be aware that the new identification in Step (2) is not just a naive ``interchange".
For example, let $a$ be the vertex opposite to the chosen edge.
According to our definition, the vertices of adjacent triangles previously glued to $a$ will not be glued to $a$.
Finally, we need to warn readers that this operation may alter the topology of the surface as we shall see.
We denote the new triangulation (of a possibly new surface) by $\mb{T}^{\t,\e}$.

\begin{lemma}\cite[Lemma 3.3]{FW} There is a sequence of mutations $\bs{\mu}^{\t,\e}$ such that  
	$$\bs{\mu}^{\t,\e}(\Diamond_l(\mb{T}))=\Diamond_l(\mb{T}^{\t,\e}).$$
\end{lemma}
\noindent Readers can find the definition of $\bs{\mu}^{\t,\e}$ in the proof of \cite[Lemma 3.3]{FW}.
Here we provide a picture for the hive quiver $\Delta_5$ (Figure \ref{f:hive}) after applying the twist for the vertex $\e = \mch{A}_1$.
$$\hivefivetwisted$$

We apply this operation to the seed transferring from $k[X_{4,0}]$.
The mutated quiver, denoted by $\Diamond_l$ \footnote{Unfortunately we use this notation for a slightly different quiver in \cite{Fk1}}, is the ice quiver corresponding to the triangulation in Figure \ref{f:pair} (right).
We denote such a cluster still by $\b{s}_l$.
This seed can be visualized by the following picture. Recall that we have defined the seed for each triangle.
Readers may reflect for a while to convince that the semi-invariants do agree along the common edge.
\begin{figure}[H] $$\mthree{{\mc{A}_2}}{\mat{\mch{A}_1\\\mc{A}_1}}{\mat{\mc{A}_1\\ \mch{A}_1}}{{\mch{A}_2}}{$\circlearrowleft$}{$\circlearrowright$}{}{}$$  \caption{} \label{f:X22}  \end{figure}

\begin{lemma} \label{L:CSX22} The coordinate algebra $k[X_{2,2}]$ is isomorphic to the upper cluster algebra $\uca(\Diamond_l,\b{s}_l)$.
	Moreover, $k[X_{2,2}^\circ]$ is exactly $k[X_{2,2}]$ localizing at all frozen variables.
\end{lemma}

\subsection{Cluster Structure of $k[\zllm]$}
We fix an element $\bs{g}:=(g_1,g_2,\dots,g_m)\in G^m$. We shall write $g_I^{-1}:=(g_I)^{-1}$.
We mainly concern the two kinds of triangles in Figure \ref{f:oddeven} for $n\in[2,m]$ and their duals.
\begin{figure}[!h] $\mtwoup{\mc{A}}{g_{[n-1]}\mch{A}}{\mc{A}g_{[n]}^{-1}}{$\circlearrowleft$}{}{}{}{$n$ is even}$ \hspace{1in} $\mtwoup{\mc{A}}{\mc{A}g_{[n-1]}^{-1}}{g_{[n]}\mch{A}}{$\circlearrowleft$}{}{}{}{$n$ is odd}$ \caption{The triangles $\t_n$} \label{f:oddeven}  \end{figure}
By definition, the dual of such a triangle changes all $g_I\mch{A}$ to $\mc{A}g_{\invarr{I}}$ and all ${\mc{A}}g_I^{-1}$ to $g_{\invarr{I}}^{-1}\mch{A}$, and reverses the orientation.
We denote the both kinds by $\t_n$ and their duals by $\check{\t}_n$.

It is clear that we can (consistently) glue the right edge of $\t_n$ to the left edge of $\t_{n+1}$ for any $n$.
We can also (non-consistently) glue $\t_n$ and $\check{\t}_n$ along their bases if we introduce the $G$-action on $\mc{A}^p\times \mch{A}^q$.
We justify this by Figure \ref{f:X22} and the following picture (for even $n$),
where we transfer the triangles to equivalent ones by the $G$-action.   
\begin{figure}[!h]
\begin{align*} 
&\mtwoup{\mc{A}}{g_{[n-1]}\mch{A}}{\mc{A}g_{[n]}^{-1}}{$\circlearrowleft$}{}{}{}{} & \bigsimeq & \qquad \mtwoup{\mc{A}g_{[n-1]}}{\mch{A}}{\mc{A}g_{n}}{$\circlearrowleft$}{}{}{}{} \\[-0.2in]
&\mtwodown{\mch{A}}{\mc{A}g_{\invint{n-1}}}{g_{\invint{n}}^{-1}\mch{A}}{$\circlearrowright$}{}{}{} & \bigsimeq & \qquad \mtwodown{g_{\invint{n}}\mch{A}}{\mc{A}g_{n}}{\mch{A}}{$\circlearrowright$}{}{}{} 
\end{align*} \caption{} \label{f:gluedual}
\end{figure}
According to construction of the last subsection, 
it is clear from the right-hand side of Figure \ref{f:gluedual} that when $\bs{g}$ varies the cluster variable $\b{s}_l(\t_n;i,j)$ associated to $\t_n$ is equal to $s\ijn$ and similarly $\b{s}_l(\check{\t}_n;i,j)=\check{s}\ijn$.

We define the {\em $G$-type} of the edge $\mc{A}g_1  ^{\line(1,0){25}} g_2\mch{A}$ to be $g_1g_2$.
Two such edges can be identified as in Figure \ref{f:X22} if and only if they have the same $G$-type.
To define the $G$-type of the edge $\mc{A}g_1 ^{\line(1,0){25}} \mc{A}g_2$, we need to fix a direction, say $\mc{A}g_1 \!\!\longrightarrow\!\! \mc{A}g_2$.
In this case, we define its $G$-type to be $g_1g_2^{-1}$.
We remark that this direction has nothing to do with the orientations of the triangles.
Then two such edges can be identified along the chosen directions if and only if they have the same $G$-type.
Although there is an ambiguity for this $G$-type, the triangle $[g_0\check{A},\mc{A}g_1,\mc{A}g_2]$ is uniquely determined up to the $G$-equivalence by the $G$-types of its edges (Indeed, the $G$-type $h_3$ of $\mc{A}g_1,\mc{A}g_2$ satisfies $h_3=h_1h_2^{-1}$ or $h_3=h_2^{-1}h_1$ where $h_i$ ($i=1,2$) is the $G$-type of $\mc{A}g_i,g_0\check{A}$).
So to efficiently draw this kind of triangulation, it makes sense to introduce the edge-labeled triangles, that is, triangles with edges labeled by their $G$-types. 
If the $G$-type is $g_{[n]}$ or $g_{\invint{n}}$, we will just label it by $[n]$ or $\overleftarrow{[n]}$.
For example, after appropriate choice of directions, here is a picture for $\t_n$ together with its adjacent triangles.
$$\glueedge$$


\begin{definition} We define the oriented triangulation $\mb{T}_m$ as follows.
The triangles in $\mb{T}_m$ are $\t_n$ ($n\in[2,m]$) and their duals.
We glue them as the recipe above. 
We have seen that all elements in $\b{s}_l(\t_n)$ agree on common edges.
So we obtained a set $\b{s}_l(\mb{T}_m)\subset k[\zllm]$ indexed by the vertices of $\Diamond_{l}(\mb{T}_m)$. 
We denote $(\Diamond_l(\mb{T}_m),\b{s}_l(\mb{T}_m))$ by $(\Diamond_l^m,\b{s}_l^m)$. 
\end{definition}
\noindent Here is a picture for $\mb{T}_5$, where all edges with the same label are identified.
Note that the leftmost diamond itself have two pairs of edges identified. It corresponds to the quiver of Figure \ref{f:twoglue} (right).
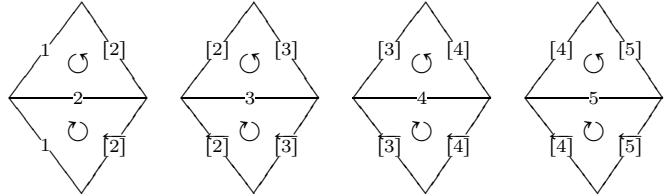
$$\mthreeq{1}{[2]}{2}{1}{\invint{2}}{$\circlearrowleft$}{$\circlearrowright$}\quad \mthreeq{[2]}{[3]}{3}{\invint{2}}{\invint{3}}{$\circlearrowleft$}{$\circlearrowright$} \quad \mthreeq{[3]}{[4]}{4}{\invint{3}}{\invint{4}}{$\circlearrowleft$}{$\circlearrowright$} \quad \mthreeq{[4]}{[5]}{5}{\invint{4}}{\invint{5}}{$\circlearrowleft$}{$\circlearrowright$}$$
We call such a picture the {\em diamond view} of $\mb{T}_m$. An edge labeled by $n\in [2,m]$ is called a {\em diamond diagonal}. 
We label the diamonds by their diagonals.
A diamond is called odd (resp. even) if its diagonal is labeled by an odd (resp. even) number.
This view is related to the map $\pi=(\pi_2,\pi_3,\dots,\pi_m)$ in Lemma \ref{L:embedding}.
To be more precise, let $\Diamond_l^{(n)}$ be the diamond $n$ and $\b{s}_l^{(n)}$ be the restriction of $\b{s}_l^m$ to $\Diamond_l^{(n)}$, then
the upper cluster algebra $\uca(\Diamond_l^{(n)},\b{s}_l^{(n)})$ is equal to $\pi_n^*(k[X_{2,2}])$.

\begin{lemma} \label{L:indep} The set $\b{s}_l^m$ is algebraically independent over $k$.
\end{lemma}
\begin{proof} By Lemma \ref{L:embedding} the functional field of $\zllm$ is isomorphic to that of the image of $\pi$.
So $k(\zllm)=\pi^*k((X_{2,2})^{m-1})$ is generated by $\bigcup_{n\in[m]} \b{s}_l^{(n)} = \b{s}_l^m$ as a field.
The Krull dimension of the integral domain $k[\zllm]$ is equal to 
$$2\dim\mc{A}+(m-2)\dim G=(l-1)(l+2)+(l^2-1)(m-2).$$
By a simple counting, this is exactly the cardinality of $\b{s}_l^m$.
Hence $\b{s}_l^m$ must be algebraically independent over $k$.
\end{proof}

We can also first glue all $\t_n$, and then all duals $\check{\t}_n$ together as shown below.
\begin{figure}[!h]
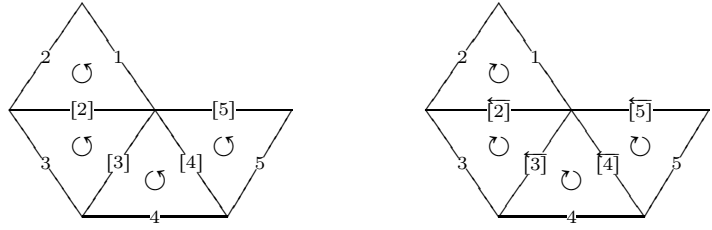
$$\mfive{}{[2]}{[3]}{[4]}{[5]}{$\circlearrowleft$}{$\circlearrowleft$}{$\circlearrowleft$}{$\circlearrowleft$}\qquad \mfive{}{\invint{2}}{\invint{3}}{\invint{4}}{\invint{5}}{$\circlearrowright$}{$\circlearrowright$}{$\circlearrowright$}{$\circlearrowright$}$$
\caption{Disk view of $\mb{T}_5$} \label{f:diskview}
\end{figure}
Such a picture is called the {\em disk view} of $\mb{T}_m$. 
We call an edge labeled by $[n]$ or $\overleftarrow{[n]}$ ($n\in [m-1]$, including $1=[1]=\overleftarrow{[1]}$) a {\em disk diagonal}.
We call the edge labeled by $[m]$ or $\overleftarrow{[m]}$ the {\em boundary}.  
It is clear from both views that topologically we get a sphere with a disk cut out.
It has $m-1$ punctures and $2$ marked points on the boundary.
However, as we shall see at the end of this subsection if the non-consistent gluing is forbidden, then it should be a (unpunctured) $(m-2)$-torus with some marked boundary circle(s).

Let $D_{\rm h}$ be the product of all cluster variables on the disk diagonals.
\begin{corollary} \label{C:local=} The upper cluster algebra $\uca(\Diamond_l^m,\b{s}_l^m)$ is a subalgebra of $k[\zllm]$.
	Moreover, the two algebras are equal after localizing at $D_{\rm h}$. 
\end{corollary}

\begin{proof} Since each $s\ijn$ or $\check{s}\ijn$ is irreducible and $k[\zllm]$ is a UFD (\cite[Theorem 3.17]{PV}), 
	$\b{s}_l^m$ is a CR1 seed in $k[\zllm]$. 
	So we have from Lemma \ref{L:RCA} the containment $k[\zllm] \supseteq \uca(\Diamond_l^m,\b{s}_l^m)$.
	By Lemma \ref{L:reduce2frozen}, $\uca(\Diamond_l^m,\b{s}_l^m)_{D_{\rm h}}$ is the same as $\uca\big((\Diamond_l^m)^{d},\b{s}_l^m\big)_{D_{\rm h}}$
	where $(\Diamond_l^m)^{d}$ is obtained from $\Diamond_l^m$ by freezing all vertices on the disk diagonals.
	Then $\uca(\Diamond_l^m,\b{s}_l^m)_{D_{\rm h}}$ contains  $\uca\big(\Diamond_l^{(n)},\b{s}_l^{(n)}\big)$ for all $n\in [2,m]$.
	Recall that by construction each $\uca\big(\Diamond_l^{(n)},\b{s}_l^{(n)}\big)$ is equal to $\pi_n^*(k[X_{2,2}])$.	
	By Lemma \ref{L:generic-by-edge} $k[\zllm]_{D_{\rm h}}$ is exactly $k[(\zllm)_\pi^\circ]$.
	Then by Lemma \ref{L:embedding} and \ref{L:CSX22}, 
	$$k[\zllm]_{D_{\rm h}}=\pi^*\left(k\big[(X_{2,2}^\circ)^{m-2}\times X_{2,2}^\circm\big]\right),$$ 
	which is generated by all $\pi_n^* \left(k[X_{2,2}]\right)$ and inverse cluster variables on the disk diagonals.
	We thus get the other containment ``$\subseteq$".
\end{proof}

Let $\d=\br{bc}$ be the common edge of two adjacent triangles in $\mb{T}$ as in Figure \ref{f:flip} (left).
A {\em flip} along $\d$ is an operation on $\mb{T}$ sending the pair to Figure \ref{f:flip} (right), 
\begin{figure}[!h]
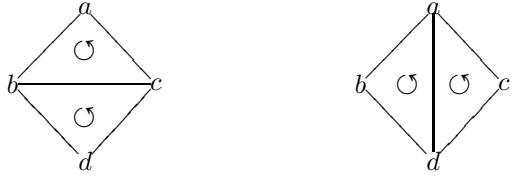
 $\mthree{a}{b}{c}{d}{$\circlearrowleft$}{$\circlearrowleft$}{}{} \hspace{1in} \mthreeflip{a}{b}{c}{d}{$\circlearrowleft$}{$\circlearrowleft$}$ \caption{The flip} \label{f:flip} \end{figure}
and keeping the remaining triangles unchanged.
We denote the new triangulation by $\mb{T}'$.
It is well-known \cite{FST} that any two consistent triangulations of $\mc{S}$ are related by a sequence of flips.
We warn readers that the flip is not defined for a pair of triangles as in Figure \ref{f:pair} (right).
\begin{lemma}\cite[10.3]{FG} There is a sequence of mutations $\bs{\mu}^{\d}$ such that 
	$$\bs{\mu}^{\d}(\Diamond_l(\mb{T}))=\Diamond_l(\mb{T}').$$
\end{lemma}

We define an edge-labeled oriented triangulation $w_0(\mb{T}_m)$ as follows.
We modify all edge-labels of $\mb{T}_m$ by the action of $w_0$.
We must clarify what we mean by the action of $w_0$ on edge-labels:
$$w_0(g_{[n]}) = w_0\big(g_1g_2^{-1}\cdots g_n^{(-1)^{n-1}}\big)= g_{w_0(1)}g_{w_0(2)}^{-1}\cdots g_{w_0(n)}^{(-1)^{n-1}},$$
which is not $[w_0(n)]$ or $[w_0(n),w_0(1)]$ in general.
We also define $\overleftarrow{w}_0(\mb{T}_m)$ as the edge-labeled triangulation obtained from ${w}_0(\mb{T}_m)$ by reversing the order of all edge-labels.
We can easily verify that when $m$ is odd
$$\overleftarrow{w}_0(g_{[n]})=\begin{cases}g_{-[m-n+1,m]} & n \text{ is even;}\\ g_{[m-n+1,m]} & n \text{ is odd,} \end{cases}$$
where we set $g_{-I}:=g_{i_1}^{-1}g_{i_2}\cdots g_{i_n}^{(-1)^{n}}$.
For example half of the disk views of $w_0(\mb{T}_4)$ and $\overleftarrow{w}_0(\mb{T}_5)$ are display below (cf. Figure \ref{f:diskview}).
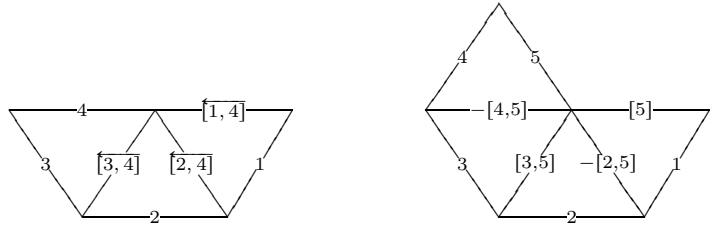
$$\mfourw{4}{\invarr{[3,4]}}{\invarr{[2,4]}}{\invarr{[1,4]}} \qquad \mfivew{$\circlearrowleft$}{[5]}{-[2,5]}{{[3,5]}}{-[4,5]}$$

\begin{lemma} \label{L:Tmflip} Apply a sequence of flips at the disk diagonals $[m-1],\dots,[3],[2]$ and their duals to $\mb{T}_m$, then we get the triangulation $\mb{T}_m^\dagger={w}_0(\mb{T}_m)$ if $m$ is even or $\mb{T}_m^\dagger=\overleftarrow{w}_0(\mb{T}_m)$ if $m$ is odd.
\end{lemma}

\begin{proof} 
	Let $\t_n^\dagger$ be the triangle obtained from $\t_n$ by replacing $\mc{A}$ by $\mc{B}=g_{[m]}^{-1}\mch{A}$ if $m$ is even or $\mc{B}=g_{[m]}\mch{A}$ if $m$ is odd.
	We can easily prove by induction that $\mb{T}_m^\dagger$ consists of triangles $\t_n^\dagger$ ($n\in[2,m-1]$) and 
	$\t_{m}^\dagger:=[\mc{B},g_1\mch{A},\mc{A}]$.
	The common edge of $t_{n}^\dagger$ and $t_{n+1}^\dagger$ is $\mc{B}^{\line(1,0){25}}\mc{A}g_{[n]}^{-1}$ if $n$ is even 
	or $\mc{B}^{\line(1,0){25}}g_{[n]}\mch{A}$ if $n$ is odd.
	They all have desired $G$-types. The situation is similar for dual triangles.
\end{proof}

\begin{proposition} \label{P:Zllmcluster} 
We have an isomorphism $k[\zllm]\cong \uca(\Diamond_l^m,\b{s}_l^m)$.
\end{proposition}

\begin{proof} By Corollary \ref{C:local=}, we have that $k[\zllm]_{D_{\rm h}} = \uca(\Diamond_l^m,\b{s}_l^m)_{D_{\rm h}}$.
The same arguments can show that the equality also holds for the seed $(\Diamond_l^m,\b{s}_l^m)^\dagger$ corresponding to $\mb{T}_m^\dagger$, that is, $k[\zllm]_{D_{\rm h}^\dagger} = \uca(\Diamond_l^m,\b{s}_l^m)^\dagger_{D_{\rm h}^\dagger}$.	
	Moreover, by Lemma \ref{L:Tmflip} and \cite[Lemma 4.14]{FW} the two seeds are mutation-equivalent so that $\uca(\Diamond_l^m,\b{s}_l^m)=\uca(\Diamond_l^m,\b{s}_l^m)^\dagger$.
	We notice that $\mb{T}_m$ and $\mb{T}_m^\dagger$ do not have any disk diagonal in common.
	Let $X:=\Spec\big(k[\zllm]\big)$ and $Y:=\Spec \left(\uca(\Diamond_l^m,\b{s}_l^m) \right)$.
	We have a regular birational map $X \to Y$ induced by the inclusion of Corollary \ref{C:local=}.
	To show it is in fact an isomorphism, it suffices to show according to \cite[Corollary 1 of II.4.4]{Sh} that
	$Z(D_{\rm h}) \cap Z(D_{\rm h}^\dagger)$ has codimension 2 in $X$.
	Recall that each function in $\b{s}_l(\mb{T}_m)$ and $\b{s}_l(\mb{T}_m^\dagger)$ is irreducible.
	Since $X$ is factorial (\cite{PV}) and all irreducible factors in $D_{\rm h}$ and $D_{\rm h}^\dagger$ are distinct,
	we conclude that the intersection has codimension 2.
\end{proof}
\noindent The above proof is similar to that of \cite[Theorem 4.15]{FW}.
As remarked there, condition for $Y$ required in \cite[Corollary 1 of II.4.4]{Sh} is being nonsingular, 
but we see from the proof that it can be weakened as being locally factorial.

In the end, we make a connection to the Fock-Goncharov's construction \cite{FG}, in which non-consistent gluing is not allowed.
\begin{proposition} If we apply the twist to each dual triangle $\check{\t}_n$ w.r.t the disk diagonal, then we transforms $\Diamond_l^m$ to $\Diamond_l(\mb{T}_{m}')$,
where $\mb{T}_{m}'$ is a consistent ideal triangulation of a (unpunctured) torus of genus $g=m-2$ with 
$$\begin{cases} \text{two boundary components and one marked point on each} & \text{ if $m$ is even};\\
\text{one boundary component and two marked points} & \text{ if $m$ is odd}.
\end{cases}$$
\end{proposition}

\begin{proof} It is not hard to show by induction that after applying such a sequence of twists, 
	we get the same triangulation of dual disk but it will be glued to the other disk differently.
	We note that all gluing is consistent now. 
	Then we glue the two disk pieces along the edge labeled by $1$.
	We obtained a polygonal presentation of a $(m-2)$-torus with a boundary on a $2m$-gon.
	Moreover, we see that when $m$ is odd two edges $[m]$ and $\overleftarrow{[m]}$ are glued to a (boundary) circle, 
	and each vertex is identified to a marked point on the boundary.
	When $m$ is even, $[m]$ (or $\overleftarrow{[m]}$) is glued to itself and form a boundary circle.
	Here is a picture for $m=4$.
	$$\gluetorus$$
\end{proof}
\noindent We remark that after applying the twists the cluster variables in the interior of each dual triangle are no longer given by the recipe described in Section \ref{ss:X22}.

We label the vertex $(i,j)$ in $\Delta_l(\t_n)$ (resp. $\Delta_l(\check{\t}_n)$) by $(\ijn)$ (resp. $(_{i,j}^{n\vee})$).
Throughout the paper, we shall use the notation $(_{i,j}^{n\veebar})$ for $(_{i,j}^{n})$ or $(_{i,j}^{n\vee})$, and $\check{v}$ for $(_{i,j}^{n\vee})$ if $v=(_{i,j}^{n})$.
Since vertices on the edges will be identified, it is good to resolve the ambiguity.
We shall use \begin{enumerate}
	\item $(_{i,0}^n)$ instead of $(_{i,0}^{n\vee})$ for each $n$;
	\item $(_{0,j}^{n-1\veebar})$ instead of $(_{0,j}^{n\veebar})$ for even $n$;
	\item $(_{i,j}^{n-1\veebar})$ instead of $(_{i,j}^{n\veebar})$ for $i+j=l$ and odd $n$.
\end{enumerate}
\noindent Readers can find a quiver in such a labeling in Figure \ref{f:diamond444p}.

In particular, we have an exchange relation in $k[\zllm]$ at each mutable vertex in $(\Diamond_l^m,\b{s}_l^m)$.
We record these exchange relations at edge vertices, which are needed in Section \ref{ss:liftseed}. 
We set $s_{0,0}^n=s_{l,0}^n=1$ in \eqref{eq:exrel1} below.
\begin{corollary} There are exchange relations:\begin{align} 	
	\label{eq:exrel1} s_{i,0}^n {s_{i,0}^n}'&=s_{i,1}^n\check{s}_{i,1}^ns_{i-1,0}^n+s_{i-1,1}^n\check{s}_{i-1,1}^ns_{i+1,0}^n,\\
	\label{eq:exrel2} s_{0,j}^n {s_{0,j}^n}'&=s_{1,j}^n s_{1,j-1}^{n+1} + s_{1,j-1}^{n} s_{1,j}^{n+1} && \text{$n$ is odd,}\\
	\label{eq:exrel3} s\ijn {s\ijn}'&=s_{i-1,j}^n s_{i,j-1}^{n+1} + s_{i,j-1}^{n} s_{i-1,j}^{n+1} && \text{$n$ is even and } i+j=l.
	\end{align} for some regular functions ${s\ijn}'$ on $\zllm$. \end{corollary}
\noindent  The exchange relations \eqref{eq:exrel2} and \eqref{eq:exrel3} also hold for the dual semi-invariants.  
Similar to the Schofield's construction in the proof of \cite[Lemma 4.12]{FW}, interested readers can give an explicit description for ${s\ijn}'$.

\section{$\SI_\bl(\kllm)$ as a UCA} \label{S:CSmain}
\subsection{Lifting the seed $(\Diamond_l^m,\b{s}_l^m)$} \label{ss:liftseed}
We are going to extend and lift the seed $(\Diamond_l^m,\b{s}_l^m)$ to a seed $(\wtd{\Diamond}_l^m,\wtd{\b{s}}_l^m)$ such that the cluster variables in $\wtd{\b{s}}_l^m$ are regular functions on $\Mod_{\bl}(\kllm):=\Rep_{\bl}(\kllm)/\SL_{\bl}$.
We first notice that the $m$ central determinants are regular functions on $\Mod_{\bl}(\kllm)$.
So we add them as $m$ frozen variables to $(\Diamond_l^m,\b{s}_l^m)$.
We label the frozen vertex corresponding to $\det a_n$ by $n$.
Next, we lift the cluster $\b{s}_l^m$ to a cluster $\wtd{\b{s}}_l^m$.
The natural candidate for $\wtd{\b{s}}_l^m(\t_n;i,j)$ is $\wtd{s}\ijn$.
We see directly from Definition \ref{D:liftsijn} that all $\wtd{s}\ijn$'s agree on the diagonals.
We shall prove that with this natural choice we are able to lift the exchange relations as well.

\begin{lemma} \label{L:regquotient} Let $f$ and $g$ be two $T$-homogeneous regular functions on $\Rep_{\bl}^1(\kllm)$.
	Suppose that the $T$-weights of $f$ and $g$ are $\lambda_1,\lambda_2$ with $\lambda_1-\lambda_2$ non-negative
	and $s:=f/g$ is also regular on $\Rep_{\bl}^1(\kllm)$.
Then $s$ is regular on $\Rep_{\bl}(\kllm)$.
\end{lemma}

\begin{proof} For any representation $M\in \Rep_{\bl}^\circ(\kllm)$, we can write $M$ as $tN$ for some $t\in (\mb{G}_{\rm m})^m$ and $N\in \Rep_{\bl}^1(\kllm)$.
	This is possible since we are working over algebraically closed fields.
	Then $s(M)=f(tN)/g(tN)=t^{\lambda_1-\lambda_2}f(N)/g(N)$.
	Since $f/g$ is regular on $\Rep_{\bl}^1(\kllm)$, we see that it is regular on $\Rep_{\bl}^\circ(\kllm)$.
Moreover, by taking the limit $t\to 0$ we see that $s$ is regular on $Z:=\left\{M\in\Rep_\bl(\kllm)\mid M(a_n)=0,\ n\in[m] \right\}$.
	
The locus of indeterminacy of $s$ is a closed set contained in a union of irreducible subvarieties of $\Rep_{\bl}(\kllm)$ defined by vanishing of central determinants.
We note that the set $Z$ is contained in every irreducible subvariety cut out by $\det a_n$.
We conclude that the locus of indeterminacy has codimension at least two.
Since $\Rep_{\bl}(\kllm)$ is normal, the function $s$ is in fact regular everywhere.
\end{proof}

\begin{definition} A vertex is called {\em edge} if it lies on an edge of a hive.
It is called {\em corner} if it is on a corner of a hive (there are 6 corner vertices on a hive).
\end{definition}

\begin{corollary} \label{C:exrellift1}
	The exchange relations \eqref{eq:exrel1}, \eqref{eq:exrel2}, and \eqref{eq:exrel3} are satisfied by the minimal lifts $\wtd{s}\ijn$ except at corner vertices $(_{l-1,0}^n)$ for each $n$, $(_{0,l-1}^n)$ for odd $n$, and $(_{l-1,1}^n)$ for even $n$.
\end{corollary}

\begin{proof} If $v=(\ijn)$ is a non-edge vertex, then we have the exchange relation
	$$s\ijn {s\ijn}' = s_{i,j-1}^ns_{i+1,j}^ns_{i-1,j+1}^n+s_{i,j+1}^ns_{i-1,j}^ns_{i+1,j-1}^n.$$
We first check using Lemma \ref{L:Tdegree} that $\wtd{s}_{i,j-1}^n\wtd{s}_{i+1,j}^n\wtd{s}_{i-1,j+1}^n$ and $\wtd{s}_{i,j+1}^n\wtd{s}_{i-1,j}^n\wtd{s}_{i+1,j-1}^n$ have the same $\lambda$-weight equal to $3i\e_n+\sum_{p\in [n-1]} \left([\epsilon(p)]_+3j\e_p + [-\epsilon(p)]_+3(l-j)\e_p \right)$, which is greater than the $\lambda$-weight of 
$\wtd{s}\ijn$. So by Lemma \ref{L:regquotient}, there is some ${s\ijn}'$ regular on $\Rep_\bl(\kllm)$.

If $v$ is an edge vertex, then we have three kinds of exchange relations \eqref{eq:exrel1}--\eqref{eq:exrel3}.
The verification is similar. For example, for \eqref{eq:exrel2} we check using Lemma \ref{L:Tdegree} that for $j<l-1$, $\wtd{s}_{1,j}^n \wtd{s}_{1,j-1}^{n+1}$ and $\wtd{s}_{1,j-1}^{n} \wtd{s}_{1,j}^{n+1}$ have the same $\lambda$-weight greater than that of $\wtd{s}_{0,j}^n$.
\end{proof}

\noindent The exchange relations cannot be trivially lifted at the three kinds of corner vertices.
For example, when $j=l-1$ the definition of $\wtd{s}_{1,j}^n$ in \eqref{eq:exrel2} has a different formula (see Definition \ref{D:liftsijn}).
We need to make use of the $m$ new frozen variables to repair them. Here is the recipe.

\begin{corollary} \label{C:exrellift2} We have exchange relations in $\SI_\bl{\kllm}$ at those corner vertices after we add \begin{enumerate}
\item  an arrow from $(_{l-1,0}^n)$ to $n$ for each $n$;
\item  an arrow from $n$ to $(_{l-1,1}^{n\veebar})$ for even $n$;
\item  an arrow from $n$ to $(_{0,l-1}^{n\veebar})$ for odd $n\geq 3$.
\end{enumerate}
\end{corollary}

\begin{proof} (1) corresponds to the exchange relation \eqref{eq:exrel1} when $i=l-1$.
We recall that $s_{l,0}=1$ by definition. 
By lemma \ref{L:regquotient} and \eqref{eq:exrel1} it suffices to check that $\wtd{s}_{l-1,1}\wtd{\check{s}}_{l-1,1}\wtd{s}_{l-2,0}$ and $\wtd{s}_{l-2,1}\wtd{\check{s}}_{l-2,1}\det(a_{n})$ have the same $\lambda$-weight, which is greater than that of $\wtd{s}_{l-1,0}^n$. We note that $\lambda$-weight of $\det a_n$ is $l\e_n$.

(2) corresponds to the exchange relation \eqref{eq:exrel2} when $j=l-1$.
We check that $\wtd{s}_{1,l-1}^n \wtd{s}_{1,l-2}^{n+1}\det(a_{n})$ and $\wtd{s}_{1,l-2}^{n} \wtd{s}_{1,l-1}^{n+1}$ have the same $\lambda$-weight, which is greater than that of $\wtd{s}_{0,l-1}^n$.

(3) corresponds to the exchange relation \eqref{eq:exrel3} when $i=l-1$ and $j=1$.
We check that $\wtd{s}_{l-2,1}^{n} \wtd{s}_{l-1,0}^{n+1}\det(a_n)$ and $\wtd{s}_{l-1,0}^n \wtd{s}_{l-2,1}^{n+1}$ have the same $\lambda$-weight, which is greater than that of $\wtd{s}_{l-1,1}^n$.
\end{proof}
\noindent Here is picture for $l=4$. The left one is for even $n$ and the right one is for odd $n\geq 3$.
We omit $n$ in the labels of vertices and write $(i,\check{j})$ for $(_{i,j}^{n\vee})$.
\begin{figure}[!h]
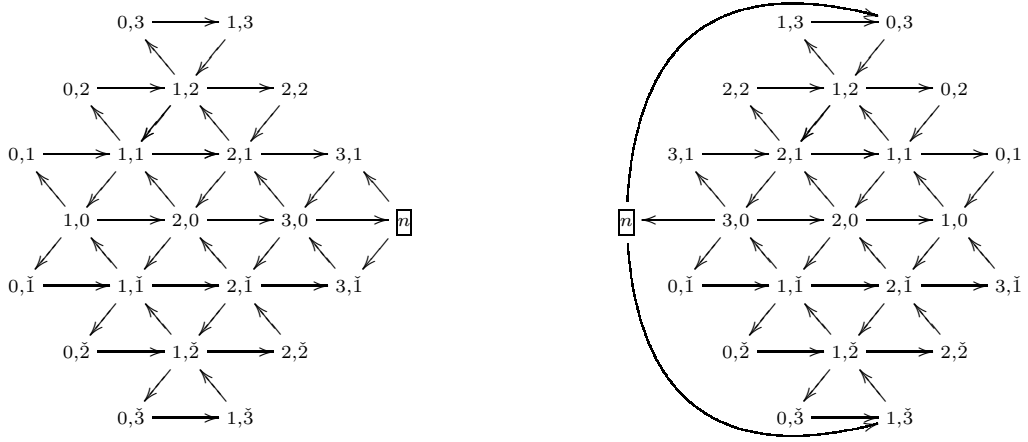
 $\diamondfourtildeeven \hspace{1in} \diamondfourtildeodd$ \caption{The lifted even diamond (left) and odd diamond (right)} \label{f:diamondtilde} \end{figure}
\noindent We also provide a (half) disk-view for $l=m=4$. The two frozen vertices $\fr{3}$ in the picture should be identified.
\begin{figure}[!h]
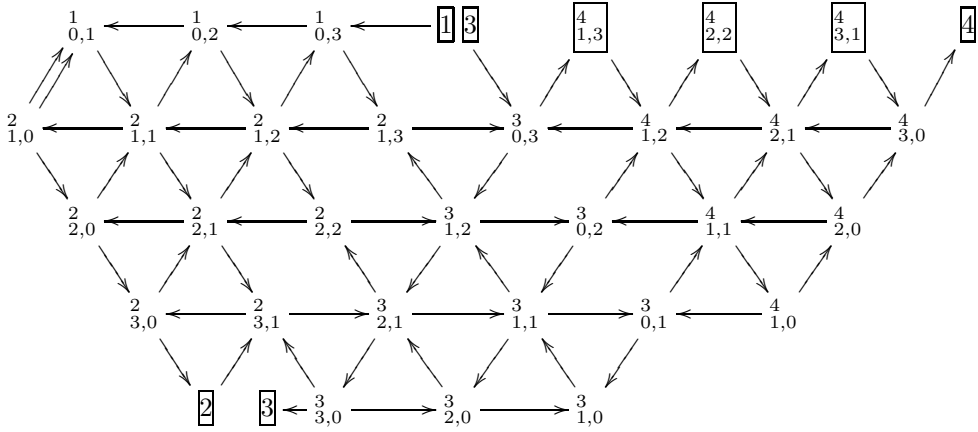
 $\hspace{-.3in} \ckronfourfourfourp$ \caption{A half of $\wtd{\Diamond}_{4}^4$} \label{f:diamond444p} \end{figure}

\begin{corollary} \label{C:RCA} The semi-invariant ring $\SI_\bl(\kllm)$ is a subalgebra of the upper cluster algebra $\uca(\wtd{\Diamond}_l^m,\wtd{\b{s}}_l^m)$.
\end{corollary} 
\begin{proof} Since $\b{s}_l^m$ is algebraically independent, we see from Lemma \ref{L:localdet} that $\wtd{\b{s}}_l^m$ is also algebraically independent. We check the conditions in Lemma \ref{L:RCA}.
The condition (1) is verified by Corollary \ref{C:exrellift1} and \ref{C:exrellift2}.
The condition (2) almost follows from Corollary \ref{C:irreducible}.
Indeed, since $\SI_\bl(\kllm)$ is a UFD, it suffices to show that each $\wtd{s}\ijn$ is not a factor of $(\wtd{s}\ijn)'$.
But this is clear from the difference of their $\wtd{\sigma}$-weights.
\end{proof}

\subsection{Cluster structure of $\SI_\bl(\kllm)$}
We first recall a lemma in \cite{FW}.
We fix a set of frozen vertices $\bs{e}$ of $\wtd{\Delta}$.
Let $\Delta$ be the full (ice) subquiver of $\wtd{\Delta}$ obtained by forgetting $\bs{e}$.
It follows from \cite[Lemma 2.2]{FW} and the remark afterwards that

\begin{lemma} \label{L:frext=}  Suppose that $B(\Delta)$ has full rank.
Then for any seed $(\wtd{\Delta},\wtd{\b{x}})$, there is some $\b{x}$ such that
	$$\uca(\wtd{\Delta},\wtd{\b{x}})_{\wtd{\b{x}}(\bs{e})}=\uca(\Delta,\b{x})\otimes k[\wtd{\b{x}}(\bs{e})^{\pm 1}].$$
\end{lemma}
\noindent The cluster $\b{x}$ can be explicitly constructed as in \cite{FW} but we do not need it here.

The next lemma can be easily proved by \cite[Lemma 3.3 and 3.4]{FW}.
\begin{lemma} \label{L:fullrank} The $B$-matrix of $\Diamond_l^m$ has full rank.	
\end{lemma}

Recall that $D_{\rm c}=\prod_{n=1}^m {s}_n$ is the product of central determinants. 
\begin{corollary} \label{C:SIcluster} We have an isomorphism 
	$\SI_\bl(\kllm)_{D_{\rm c}}\cong \uca(\wtd{\Diamond}_l^m,\wtd{\b{s}}_l^m)_{D_{\rm c}}$.
\end{corollary}

\begin{proof} Recall from Lemma \ref{L:localdet} that 
	$\SI_\bl\big(\kllm\big)_{D_{\rm c}} = k\big[\zllm\times (\mb{G}_{\rm m})^m\big]$. 
	So by Proposition \ref{P:Zllmcluster} we have that
$$\SI_\bl\big(\kllm\big)_{D_c} =\uca(\Diamond_l^m,\b{s}_l^m) \otimes k[(\mb{G}_{\rm m})^m].$$
	On the other hand, we know from Lemma \ref{L:frext=} that 
$$\uca(\wtd{\Diamond}_l^m,\wtd{\b{s}}_l^m)_{D_c}=\uca({\Diamond}_l^m,{\b{t}}_l^m)\otimes k[{s}_1^{\pm 1},\dots,{s}_m^{\pm 1}].$$
Here we use the letter $t$ to indicate that the cluster $\b{t}_l^m$ may not be the same as $\b{s}_l^m$. 
But in any case we get our desired isomorphism.
\end{proof}

\begin{lemma} \label{L:ACR} We have that $\SI_\bl(\kllm)\subseteq \uca(\wtd{\Diamond}_l^m,\wtd{\b{s}}_l^m)$.
\end{lemma}

\begin{proof} 	We already have that $\SI_\bl(\kllm)_{D_{\rm c}}\cong \uca(\wtd{\Diamond}_l^m,\wtd{\b{s}}_l^m)_{D_{\rm c}}$.
	So every element in $\SI_\bl(\kllm)$ can be written as a Laurent polynomial $p$ in any $\wtd{\b{t}}_l^m\sim \wtd{\b{s}}_l^m$.
	We need to show that $p$ is polynomial in each $s_k:=\det a_k$. Let 
	$$\mc{D}_k := \{M\in \Rep_\bl(\kllm) \mid \det(M(a_k))=0 \}.$$
	We claim that it suffices to show that $\wtd{\b{s}}_l^m\setminus s_k$ is still algebraically independent when restricted to $\mc{D}_k$.
	If this is the case, then $(\wtd{\b{t}}_l^m\setminus s_k)\mid_{\mc{D}_k}$ is also algebraically independent.
	Suppose for the contradiction that there is some $s\in \SI_\bl(\kllm)$ which is not polynomial in $s_k$.
	Then we write $s$ as $s=f/(hs_k^d)$ where $f$ (resp. $h$) is a polynomial (resp. monomial) in $\wtd{\b{t}}_l^m$ such that $s_k \notdivides f$.
	By the algebraic independence of $(\wtd{\b{t}}_l^m\setminus s_k)\mid_{\mc{D}_k}$, $f$ is not a zero function on $\mc{D}_k$.
	But $s_k$ constantly vanishes on $\mc{D}_k$. So $s$ is not regular on $\mc{D}_k$.
	
	Suppose that $k$ is odd.
	Let $Z_{l,l}^{m,k}$ be obtained from $Z_{l,l}^{m}$ by replacing the $k$-th $G$ by $D:=\{M\in {\rm Mat}_l \mid \det(M)=0\}$.
	The definition of $\b{s}_l^m$ does not involve the inverse of $g_k$ when $k$ is odd, so the definition can be straightforwardly extended to $Z_{l,l}^{m,k}$.
	It is easy to see from Lemma \ref{L:localdet} and Definition \ref{D:liftsijn} that $(\wtd{\b{s}}_l^m\setminus s_k)\mid_{\mc{D}_k}$ is algebraically independent if and only if so is $\b{s}_l^m$ as functions on $Z_{l,l}^{m,k}$.
    Moreover, each map $\pi_i$ can be defined for $Z_{l,l}^{m,k}$ as well.
	By examining the proof of Lemma \ref{L:indep}, we see that $\b{s}_l^m$ as functions on $Z_{l,l}^{m,k}$ is algebraically independent is equivalent to Lemma \ref{L:embeddingk} below, an analogue of Lemma \ref{L:embedding} for $Z_{l,l}^{m,k}$.

    If $k$ is even, then we assume $m$ is even and we apply a sequence of flips as in Lemma \ref{L:Tmflip}.
    Let $(\b{s}_l^m)^\dagger$ be the new cluster after applying the sequence of mutations.
    Lemma \ref{L:Tmflip} and \cite[Lemma 4.14]{FW} implies that $(\b{s}_l^m)^\dagger$ is obtained from $\b{s}_l^m$ by the action of $w_0$.
    In particular $(\b{s}_l^m)^\dagger$ is defined on $Z_{l,l}^{m,k}$ so the same argument applies.
    If $m$ is odd, then the desired statement follows from the case of $m+1$.
\end{proof}

Let $\hat{H}$ be the set of diagonal matrices in ${\rm Mat}_l$ with diagonal $(d_1,d_2,\dots,d_{l-1},0)$ $(d_i\neq 0)$.
For any $\hat{h}\in \hat{H}$, there is a unique $h\in H\subset G$ such that $h$ and $\hat{h}$ have the same first $l-1$ diagonal entries.
Since $(\Id_{l-1},0)h = (\Id_{l-1},0)\hat{h}$, we see that $F\hat{h} = Fh$. Similarly, $\hat{h}\check{F} = h\check{F}$.
Recall that $D:=\{M\in {\rm Mat}_l \mid \det(M)=0\}$.

\begin{lemma} The set $D^\circ:=U^-\hat{H}U^+$ is dense in $D$.	
\end{lemma}

\begin{proof} $D$ is an irreducible hypersurface in ${\rm Mat}_l$. By a simple dimension argument, 
it is enough to show that the map $(\bar{u},\hat{h},u)\mapsto \bar{u}\hat{h}u$ is injective.
Suppose that $\bar{u}_1\hat{h}_1u_1=\bar{u}_2\hat{h}_2u_2$, then $\bar{u}\hat{h}_1 = \hat{h}_2u$ where $\bar{u}=\bar{u}_2^{-1}\bar{u}_1\in U^-$ and $u=u_2u_1^{-1}\in U$.
Elementary matrix calculation show that $\bar{u}=u=e$ so that $(\bar{u}_1,\hat{h}_1,u_1)=(\bar{u}_2,\hat{h}_2,u_2)$.
\end{proof}

We extend the notion of Gauss decomposable to $D$. We say $g\in D$ is Gauss decomposable if it is in $D^\circ$.
For an odd number $k\in[m]$, let $(Z_{l,l}^{m,k})^\circ_\pi$ be the (dense) subset of $Z_{l,l}^{m,k}$ where $g_{[n]}$ and $g_{\invint{n}}$ are Gauss decomposable for each odd $n$ and $g_{[n]}\wtd{w}_0$ and $g_{\invint{n}}\wtd{w}_0$ are Gauss decomposable for each even $n$.

\begin{lemma} \label{L:embeddingk} The restriction of the map 
	$$\pi:=(\pi_2,\dots,\pi_{m}): Z_{l,l}^{m,k}\to (X_{2,2})^{m-1}$$ 
	to the dense set $(Z_{l,l}^{m,k})^\circ_\pi$ is an embedding.
\end{lemma}

\begin{proof} The proof is almost the same of Lemma \ref{L:embedding}.
	A point in $(Z_{l,l}^{m,k})^\circ$ can be represented by $(F,g_1,\dots,g_m,\check{F})$.
	Suppose that it has the same image as $(F,g_1',\dots,g_m',\check{F})$.
	In particular, they have the same image under $\pi_2: Z_{l,l}^{m,k}\to X_{2,2}$, that is,
	\begin{align*} 
	(Fg_1,Fg_2,g_2^{-1}g_1\check{F},\check{F})&=(Fg_1',Fg_2',{g_2'}^{-1}g_1'\check{F},\check{F}) \text{ in } X_{2,2}. 
	\end{align*}
	We keep in mind that \begin{equation} \label{eq:Fhhat} Fh=F\hat{h}\ \text{ and }\ h\check{F}=\hat{h}\check{F},\end{equation}
	and their stabilizers are $U^-$ and $U$ respectively.
	Since $g_1$ and  $g_1'$ are Gauss decomposable, we can assume that $g_1,g_1'\in \hat{H}$ if $k=1$, otherwise $g_1,g_1'\in H$. 
    In any case, in view of \eqref{eq:Fhhat} the same argument as in Lemma \ref{L:embedding} shows that $g_1=g_1'$ and $g_2=g_2'$. 
    The rest of the induction is conducted similarly.	
\end{proof}

Let $\wtd{\bs{\sigma}}_l^m$ be the weight configuration given by $\wtd{\bs{\sigma}}_l^m(v)$ equal to the $\wtd{\sigma}$-weight of $\wtd{\b{s}}_l^m(v)$. It follows from Corollary \ref{C:RCA} and Lemma \ref{L:ACR} that
\begin{theorem} \label{T:equal} The semi-invariant ring $\SI_\bl(\kllm)$ is isomorphic to the graded upper cluster algebra $\uca(\wtd{\Diamond}_l^m,\wtd{\b{s}}_l^m; \wtd{\bs{\sigma}}_l^m)$.
\end{theorem}

\part{The Cluster Models} \label{P:II}
\section{The QP Models} \label{S:QP}
\subsection{The QP models of UCA} \label{ss:QP}
In \cite{DWZ1} and \cite{DWZ2}, the mutation of quivers with potentials is invented to model the cluster algebras.
Following \cite{DWZ1}, we define a potential $W$ on an ice quiver $\Delta$ as a (possibly infinite) linear combination of oriented cycles in $\Delta$.
More precisely, a {\em potential} is an element of the {\em trace space} $\Tr(\ckQ):=\ckQ/[\ckQ,\ckQ]$,
where $\ckQ$ is the completion of the path algebra $k\Delta$ and $[\ckQ,\ckQ]$ is the closure of the commutator subspace of $\ckQ$.
The pair $(\Delta,W)$ is an {\em ice quiver with potential}, or IQP for short.
For each arrow $a\in \Delta_1$, the {\em cyclic derivative} $\partial_a$ on $\widehat{k\Delta}$ is defined to be the linear extension of
$$\partial_a(a_1\cdots a_d)=\sum_{k=1}^{d}a^*(a_k)a_{k+1}\cdots a_da_1\cdots a_{k-1}.$$
For each potential $W$, its {\em Jacobian ideal} $\partial W$ is the (closed two-sided) ideal in $\ckQ$ generated by all $\partial_a W$.
The {\em Jacobian algebra} $J(\Delta,W)$ is the quotient algebra $\widehat{k\Delta}/\partial W$.
If $W$ is polynomial and $J(\Delta,W)$ is finite-dimensional, then the completion is unnecessary to define $J(\Delta,W)$.
This is the situation assumed throughout the paper.

The key notion introduced in \cite{DWZ1,DWZ2} is the {\em mutation} of quivers with potentials and their decorated representations.
Since we do not need the mutation in an explicit way, we refer readers to the original text.
Unlike the quiver mutation, the mutation of IQP is not always defined for any sequence of (mutable) vertices because 2-cycles may be created along the way.
A sequence of vertices is call {\em admissible} for an IQP if its mutation along this sequence is defined. 
In this case the mutation of IQP in certain sense ``lifts" the quiver mutation.
If all sequences are admissible for $(\Delta,W)$, then we call $(\Delta,W)$ {\em nondegenerate}.
\begin{definition}[\cite{DWZ1}] A potential $W$ is called {\em rigid} on a quiver $\Delta$ if
	every potential on $\Delta$ is cyclically equivalent to an element in the Jacobian ideal $\partial W$.
	Such a QP $(\Delta,W)$ is also called {\em rigid}.
\end{definition}
\noindent It is known \cite[Proposition 8.1, Corollary 6.11]{DWZ1} that every rigid QP is $2$-acyclic, and the rigidity is preserved under mutations. In particular, any rigid QP is nondegenerate.

\begin{definition} A {\em decorated representation} of a Jacobian algebra $J:=J(\Delta,W)$ is a pair $\mc{M}=(M,M^+)$,
	where $M$ is a finite-dimensional $J$-module and $M^+$ is a finite-dimensional $k^{\Delta_0}$-module.
\end{definition}

Let $\mc{R}ep(J)$ be the set of decorated representations of $J(\Delta,W)$ up to isomorphism. 
Let $K^b(\proj J)$ be the homotopy category of bounded complexes of projective representations of $J$,
and $K^2(\proj J)$ be the subcategory of 2-term complexes in fixed degrees (say $-1$ and $0$).
There is a bijection between two additive categories $\mc{R}ep(J)$ and $K^2(\proj J)$ mapping any representation $M$ to its minimal presentation in $\Rep(J)$, and the simple representation $S_u^+$ of $k^{\Delta_0}$ to $P_u\to 0$.
We use the notation $P(\beta)$ for $\bigoplus_{v\in \Delta_0} \beta(v) P_v$, 
where $\beta\in \mb{Z}_{\geqslant 0}^{\Delta_0}$ and $P_v$ is the indecomposable projective representation corresponding to the vertex $v$.
The {\em weight vector} of a projective presentation $P(\beta_1)\to P(\beta_0)$ is equal to $\beta_1-\beta_0$.

\begin{definition} The {\em $\g$-vector} $\g(\mc{M})$ of a decorated representation $\mc{M}$ is the {\em weight vector} of its image in $K^2(\proj J)$.
\end{definition}

\begin{definition}[\cite{DF}]
	To any $\g\in\mathbb{Z}^{\Delta_0}$ we associate the {\em reduced} presentation space $$\PHom_J(\g):=\Hom_J(P([\g]_+),P([-\g]_+)),$$
	where $[\g]_+$ is the vector satisfying $[\g]_+(u) = \max(\g(u),0)$.
	We denote by $\Coker(\g)$ the cokernel of a general presentation in $\PHom_J(\g)$.
\end{definition}
\noindent Reader should be aware that $\Coker(\g)$ is just a notation rather than a specific representation.
If we write $M=\Coker(\g)$, this simply means that we take a presentation general enough (according to context) in $\PHom_J(\g)$, 
then let $M$ to be its cokernel.

\begin{definition}[{\cite{Fs1}}] \label{D:mu_supg}
	A representation is called {\em $\mu$-supported} if its supporting vertices are all mutable.
	A weight vector $\g\in K_0(\proj J)$ is called {\em $\mu$-supported} if $\Coker(\g)$ is $\mu$-supported.
	Let $G(\Delta,W)$ be the set of all $\mu$-supported vectors in $K_0(\proj J)$.
\end{definition}

For a vector $\g\in \mb{Z}^{\Delta_0}$, we write $\b{x}^\g$ for the monomial $\prod_{v\in \Delta_0} x_v^{\g(v)}$.
For $u\in\Delta_\mu$, we set ${y}_u= \b{x}^{-b_{u}}$ where $b_u$ is the $u$-th row of the matrix $B(\Delta)$,
and let ${\b{y}}=\{{y}_u\}_{u\in\Delta_\mu}$.
\begin{definition}[\cite{P}]
	We define the {\em generic character} $C_W:G(\Delta,W)\to \mb{Z}(\b{x})$~by
	\begin{equation*} \label{eq:genCC}
	C_W(\g)=\b{x}^{\g} \sum_{\e} \chi\big(\Gr^{\e}(\Coker(\g)) \big) {\b{y}}^{\e},
	\end{equation*}
	where $\Gr^{\e}(M)$ is the variety parameterizing $\e$-dimensional quotient representations of $M$, and $\chi(-)$ denotes the topological Euler-characteristic.
\end{definition}
\noindent It is known \cite[Lemma 5.3]{Fs1} that $C_W(\g)$ is an element in $\uca(\Delta)$.
Note that $C_W(\g)$ has a well-defined $\g$-vector $\g$, so we have that $G(\Delta,W)\subseteq G(\Delta)$.

\begin{theorem}[{\cite[Corollary 5.14]{Fs1}, {\em cf.} \cite[Theorem 1.1]{P}}] \label{T:GCC} Suppose that IQP $(\Delta,W)$ is nondegenerate and $B(\Delta)$ has full rank.
	The generic character $C_W$ maps $G(\Delta,W)$ (bijectively) to a set of linearly independent elements in $\br{\mc{C}}(\Delta)$ containing all cluster monomials.
\end{theorem}

\begin{definition} \label{D:model}
	We say that an IQP $(\Delta,W)$ {\em models} an algebra $\mc{A}$ if the generic cluster character maps $G(\Delta,W)$ onto a basis of $\mc{A}$.
	If $\mc{A}$ is the upper cluster algebra $\uca(\Delta)$, then we simply say that $(\Delta,W)$ is a {\em cluster model}.
	If in addition $G(\Delta,W)$ is given by lattice points in some polyhedron, then we say that the model is {\em polyhedral}.
\end{definition}

\noindent 
\cite[Proposition 5.15]{Fs1} implies that being a (polyhedral) cluster model is mutation-invariant.

\begin{definition}[\cite{GHKK}]
	We say that a (frozen or mutable) vertex $e$ can be {\em optimized} in $\Delta$ if there is a sequence of mutations away from $e$ making $e$ into a sink or source of $\Delta$ (possibly after deleting arrows between frozen vertices). 
	It can be optimized in an IQP $(\Delta,W)$ if in addition such a sequence is admissible.
\end{definition}

\begin{theorem}[{\cite[Theorem 2.13]{FW}}] \label{T:Cmodel} Let $\wtd{W}$ be any potential on $\wtd{\Delta}$ such that its restriction on $\Delta$ is $W$.
	Suppose that $B(\Delta)$ has full rank, and each vertex in $\bs{e}$ can be optimized in $(\wtd{\Delta},\wtd{W})$.
	If $(\Delta,W)$ is a (polyhedral) cluster model, then so is $(\wtd{\Delta},\wtd{W})$.
\end{theorem}

\subsection{Some sequences of mutations} \label{ss:museq}
The following lemma can be easily verified by the involutive property of the mutation.
\begin{lemma} \label{L:delrecover} Let $\Delta$ be an ice quiver and $e$ be a vertex of $\Delta$. 
Let $\bs{\mu}$ be a sequence of mutations away from $e$.
We apply $\bs{\mu}$ to $\Delta$, delete $e$, and then apply $\bs{\mu}$ backwards.
We end with the ice quiver $\Delta$ with $e$ deleted.
\end{lemma}

Consider the following hive strip of length $n+1$. The vertex $n$ can be frozen or mutable.
$$\hiverow{\ss 0}{\ss 1}{\ss{n-1}}{\ss n}$$

\begin{lemma} \label{L:hiverow} The sequence of mutations $(0,1,\cdots,n-1)$ transforms the quiver 
	into $$\hiverowmuhalf{\ss 0}{\ss 1}{\ss{n-1}}{\ss{n}}$$
	In particular, the sequence optimizes the (frozen or mutable) vertex $n$.
\end{lemma}

Consider any odd diamond in $\wtd{\Diamond}_l^m$ (Figure \ref{f:diamondtilde} (right)). 
If we apply the twist for $\t=$ each triangle in the diamond and $\e=$ the diamond diagonal, then we transform this diamond to a new one.
Here is the quiver of the new diamond for $l=4$. We suggest readers comparing it with Figure \ref{f:diamondtilde} (right).
$$\diamondfourtildeoddmu$$
There are some additional arrows between edge vertices because the gluing is now non-consistent (see Figure \ref{f:Pair} (right)).
Note that the even diamonds in $\wtd{\Diamond}_l^m$ are already in a similar form (see Figure \ref{f:diamondtilde} (left)).
Finally, we mutate along the diamond diagonal towards $n$.
In this way, we optimize $n$ by Lemma \ref{L:hiverow}. We conclude that
\begin{lemma} \label{L:optdet} Each frozen vertex $n$ in $\wtd{\Diamond}_l^m$ can be optimized.
\end{lemma}

We apply such twists for each odd diamond, and we denote the new triangulation by $\mb{T}_m'$.
Note that all triangles are glued non-consistently in $\mb{T}_m'$.
We denote the quiver $\Diamond_l(\mb{T}_m')$ by $\Diamondsh_l^m$.
We apply the corresponding sequence of mutations to the graded seed $(\wtd{\Diamond}_l^m,\wtd{\b{s}}_l^m; \wtd{\bs{\sigma}}_l^m)$, and denote the new seed by $(\br{\Diamond}_l^m,\br{\b{s}}_l^m; \br{\bs{\sigma}}_l^m)$.
Note that $\br{\Diamond}_l^m$ can be obtained from $\Diamondsh_l^m$ by adding $m$ frozen vertices.
$\br{\bs{\sigma}}_l^m$ can be directly computed from $\wtd{\bs{\sigma}}_l^m$.
For the rest of this paper, we will mainly deal with the ice quiver $\br{\Diamond}_l^m$ rather than $\wtd{\Diamond}_l^m$ because this makes proofs simpler.

\begin{definition} Let $\Delta$ be a quiver and $\b{u}$ a set of vertices of $\Delta$.
	The {\em minimal extended subquiver of $\b{u}$ in $\Delta$} is the smallest (not necessarily full) subquiver of $\Delta$ containing all arrows coming out and going into $\b{u}$.
\end{definition}

We come back to the hive strip of length $n+1$.
\begin{lemma} \label{L:longrow} Let $I=(i_1,i_2,\cdots,i_l)\in [n-1]$ be an increasing sequence.
	We first mutate along $I$, then along $[n-1]\setminus I$ in the increasing order. 
Then we obtain the following minimal extended subquiver of $\{0,n\}$.
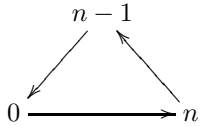
$$\minext{n-1}{0}{n}$$
\end{lemma}
\begin{proof} We shall prove by induction on the length of $I$. We set $i_0=0$ and $i_{l+1}=n$.
	Note that the sequence of mutations is equivalent to the sequence $(i_1,a_1,i_2,a_2,\cdots,i_l,a_l,a_{l+1})$,
	where $a_k$ is the sequence $(i_{k-1}+1,i_{k-1}+2,\cdots,i_k-1)$.	
	Here is the quiver after the mutation at $i_1$.
	$$\hiverowvari{\ss 0}{\ss{i_1-1}}{\ss i_1}{\ss{i_1+1}}{\ss n}$$	
	Then by a similar induction as in the proof of Lemma \ref{L:hiverow}, we can show that 
	the minimal extended subquiver $\rho_1$ of $[i_1+1,n-1]\cup \{0,n\}$ is the following.
$$\hiverowvarmu{\ss{i_1(-1)}}{\ss{0}}{\ss{i_1+1}}{\ss{i_1+2}}{\cdots}{\ss{n}}$$
The vertex $i_1(-1)$ is $i_1-1$ if $i_1-i_0>1$; otherwise it is $i_1$.
	Since the later mutations are all in $[i_1+1,n-1]$, this is the only relevant part.
Although the shape of this quiver is slightly different with the original one,
the minimal extended subquiver $\rho_2$ of $[i_2+1,n-1]\cup \{0,n\}$ after mutating along $(i_2,a_2)$
is of the same shape but of shorter length (just replacing $i_1$ by $i_2$ in the above picture).
Similarly the vertex $i_2(-1)$ is $i_2-1$ if $i_2-i_1>1$; otherwise it is $i_2$.
We can thus prove by induction that the same holds for the minimal extended subquiver $\rho_{l}$ of $[i_{l}+1,n-1]\cup\{0,n\}$ after mutating along $(i_{l},a_{l})$.

In the end, the last sequence of mutations $a_{l+1}$ transformations the quiver $\rho_{l}$ to
$$\hiverowvarlast{\ss{n}}{\ss 0}{\ss{i_l+1}}{\ss{i_l+2}}{\cdots}{\ss{n-1}}$$
This is what we desired.
\end{proof}

\begin{lemma} \label{L:crosscommu} Consider the following cross of two hive strips.
	$$\crossmu$$
We first mutate at the center $i$, then the mutation at any other central vertex (that is non-dot vertex) on one strip commutes with the mutation at any other central vertex on the other strip.
\end{lemma}

\begin{proof} Suppose that we have mutated at $i$.
A trivial observation is that it suffices to check that the mutation at either $i-1$ or $i+1$ commutes with the mutation at either $(i-1)'$ or $(i+1)'$.
This can be directly verified, say by \cite{Ke}.
\end{proof}

Each frozen vertex $v=(_{i,j}^m)$ generates a path $p_v$ in ${\Diamond}_l^m$ as follows.
For the sake of explanation, we describe such a path with all arrows reversed.
It starts from $v_0:=v$, going in the only straight direction until it hits an edge vertex $v_1$ in $\Diamond_l^{(m-1)}$.
Then it starts from $v_1$, going in the only straight direction until it hits an edge vertex in the previous diamond.
Keep this way until it hits the vertex $(_{0,j}^1)$.
Then from $(_{0,j}^{{1}\vee})$ it follows the same rule until it hits a frozen vertex $v^*$.
It is easy to see that $v^*=(_{j,i}^m)$ if $m$ is even and $v^*=\check{v}$ if $m$ is odd.
We see that $v$ and $v^*$ coincide only when $m$ is even and $i=j$ (in particular $l$ is even as well).

\begin{example} \label{ex:cross} Here is a picture for such a path when $l=5$ and $m=4$.
\begin{figure}[H] $\diamondfivecrleft \diamondfivecrmid \diamondfivecrright$ \caption{} \label{f:diamond54} \end{figure}
\end{example}

\begin{lemma} \label{L:optfr} Let $v$ be a frozen vertex of $\Diamond_l^m$ and $p_v$ be the path defined above.
We assume that $v\neq v^*$. 
We first mutate at each cross-vertex of $p_v$ for $v=(_{i,j}^m)$, then mutate at each non-cross vertex along this path.
In this way, we optimize the frozen vertex $v$ (in fact both $v$ and $v^*$).
\end{lemma}

\begin{proof} By Lemma \ref{L:longrow} and \ref{L:crosscommu}, the minimal extended subquiver of $\{v,v^*\}$ is given by
	$$\fr{v}\to u \to \fr{v^*},$$
	after deleting the arrows between the frozen vertices $v$ and $v^*$.
	Here $u$ is the last vertex connected to $v$ on the path $p_v$. 
\end{proof}

\begin{remark} The above lemma does fail if $v=v^*$.
In the smallest non-trivial case when $l=m=2$. The quiver $\Diamond_2^2$ is 
\begin{equation} \label{eq:d22} \ckrontwotwotwo \end{equation}
The quiver has finite mutation-type so it can be easily checked that none of the two frozen vertices can be optimized. 
\end{remark}

Now we proceed to the second part of this subsection. We are going to deal with the mutable vertices of $\Diamond_l^m$.
Recall that $\Diamond_l$ is the ice quiver corresponding to the triangulation of Figure \ref{f:pair} (right).
We delete the (frozen) vertices on the two right edges, and denote this ice quiver by $\hat{\Diamond}_l$.
\begin{lemma}[cutting diamond] \label{L:cuthalf} Using mutations and sink-source deletions, we can delete all the vertices on the diagonal of $\hat{\Diamond}_l$.
\end{lemma}

\begin{proof} We prove by induction. By Lemma \ref{L:hiverow}, we can delete the rightmost vertex.
	Now suppose that we have deleted a few vertices from the right as shown below.
	The vertices in gray has been deleted.
$$\bigdiamond$$ 
Let us look at the subquiver in red. We shall denote the sequence of mutations on the $r$-th row 
by 
$$\bs{\mu}_{r}:=\mu_{1,r}\mu_{2,r}\cdots \mu_{k,r}.$$
We also write $\bs{\mu}_r^{\veebar}$ for $\bs{\mu}_{\check{r}} \bs{\mu}_r$.
By Lemma \ref{L:hiverow}, after applying $\bs{\mu}_{n-k}^\veebar \cdots \bs{\mu}_{2}^\veebar \bs{\mu}_{1}^\veebar$,
we transform the subquiver in red into the following one 
$$\bigdiamondmu$$
So we can apply Lemma \ref{L:hiverow} again to optimize the vertex $(k,0)$.
Since all mutations are outside $(k,0)$, after deleting $(k,0)$ we can apply the same sequence of mutations backwards to recover the original shape (Lemma \ref{L:delrecover}).
We finish the proof by induction.
\end{proof}

\subsection{$(\br{\Diamond}_l^m,\br{W}_l^m)$ is a cluster model} \label{ss:CM}
Let $a$ be the sum of all east arrows in the diamond view of $\br{\Diamond}_l^m$.
Let $b$ (resp. $c$) be sum of all southwest (resp. northwest) arrows in the even diamonds and all northwest (resp. southwest) arrows in the odd diamonds.
These arrows are also called of type $A$, $B$, and $C$ respectively.
We remark that the arrows $_{0,j+1}^1\to_{0,j}^1$ are summands in both $b$ and $c$. 
We hope that Figure \ref{f:diamond54} is helpful for readers.
We define the potential $$\br{W}_l^m:=abc-acb.$$
We call an oriented cycle a 3-cycle if it consists of three arrows.
After a short reflection, we can see that the definition can be rephrased as follows.
$\br{W}_l^m$ is the alternating sum of all 3-cycles, except that we only take two 3-cycles from the subquiver \eqref{eq:d22}.
The Jacobian ideal $\partial(\br{W}_l^m)$ is generated by the elements
\begin{align}
\label{eq:Jrel1} &e_{u'}(bc-cb)e_u,\ e_{u'}(ca-ac)e_u,\ e_{u'}(ab-ba)e_u && \text{for $u\neq(_{0,j}^{1}),(_{1,0}^{2})$ mutable,}\\[-0.2cm]
\intertext{\hspace{3.7cm} and $u$ and $u'$ lie on the same side of diamond diagonals;} 
\label{eq:Jrel21} &(bc-cb)e_u,\ (ac-ab)e_u && \text{for $u=(_{0,j}^{1})$;}\\
\label{eq:Jrel22} &(bc-cb)e_u && \text{for $u=(_{0,1}^{1})$;}\\
\label{eq:Jrel23} &(bc-cb)e_u,\ ace_u,\ abe_u && \text{for $u=(_{1,0}^{2})$;}\\
\label{eq:Jrel3} &cae_v,\ bae_v && \text{for $v$ frozen.}
\end{align}
The requirement for $u$ and $u'$ in \eqref{eq:Jrel1} simply excludes the elements $e_{1,j-1}^{2\vee}(ab-ba) e_{1,j}^2$ and $e_{1,j-1}^{2} (ca-ac) e_{1,j}^{2\vee}$.
They should be excluded because for example $e_{1,j-1}^{2\vee} ba e_{1,j}^2$ is not contained in any 3-cycle and $e_{1,j-1}^{2\vee} ab e_{1,j}^2$ is zero.
We denote the Jacobian algebra $J(\br{\Diamond}_l^m,\br{W}_l^m)$ by $\br{J}_l^m$.

\begin{lemma} The IQP $(\br{\Diamond}_l^m,\br{W}_l^m)$ is rigid.
\end{lemma}
\begin{proof}  
We need to show that every cycle up to cyclic equivalence is zero in the Jacobian algebra $\br{J}_l^m$.
By repeatedly applying \eqref{eq:Jrel1} and cyclic equivalence, we can make any 3-cycle equivalent to $cbae_v$ or $bcae_v$ (depending on whether it is above or below the diamond diagonal) for some frozen $v$, which is zero in $\br{J}_l^m$ by \eqref{eq:Jrel3}.

For any cycle not hitting the edge $1$, it must contains all three types of arrows. 
By \eqref{eq:Jrel1} it is equivalent to a cycle containing some 3-cycle.

Now let $s$ be any cycle hitting the edge $1$, then it must pass at least two vertices on that edge. 
So $s$ contains a subpath of form $p=e_{u'} a p' b e_{u}$ or $p=e_{u'} ap'c e_{u}$ 
where $u=(_{0,j}^1)$ and $u'=(_{0,j'}^1)$; the path $p'$ does not hit the edge $1$.
Suppose that we are in the former case (the latter can be treated similarly).
If $p'$ contains a type $C$ arrow, then by relations \eqref{eq:Jrel1} and \eqref{eq:Jrel21} or \eqref{eq:Jrel22} we get a 3-cycle $acbe_{u}$.
So we can assume that $s$ contains only two types of arrows and write $s=e_{u'} a^r b^s e_{u}$. 
By \eqref{eq:Jrel1} we can move the last $a$ to the right and get $s=e_{u'} a^{r-1} b^{s-1} abe_{u}$.
If $j>1$, then $s>r>1$ and by the relation $abe_u=ace_u$ of \eqref{eq:Jrel21}
we come back to the case when $s$ contains three types of arrows.
If $j=1$, then we get $abe_u=0$ by \eqref{eq:Jrel23}. 
\end{proof}

\begin{theorem} \label{T:CM} The rigid potential IQP $(\br{\Diamond}_l^m,\br{W}_l^m)$ is a polyhedral cluster model and its Jacobian algebra is finite-dimensional.
\end{theorem}

\begin{proof} By the remark after \cite[Theorem 3.17]{Fs2}, the case for $m$ follows from the case for $m+1$. 
	So we can assume that $m\geq 3$ is odd. 
	We will recursively apply Theorem \ref{T:Cmodel}.
	But the proof will go backwards, that is, we use sink-source deletions to reduce $\br{\Diamond}_l^m$ to some quiver, which is a cluster model with its rigid potential. 
	
	Step 1. By Lemmas \ref{L:fullrank} and \ref{L:optdet}, we can delete all frozen vertices corresponding to the central determinants to obtain $\Diamondsh_l^m$. 
	Then we hope to delete the rest frozen vertices $(\Diamondsh_l^m)_\nu$ as well but we cannot do this directly because the full rank condition will be violated. We need to add some axillary frozen vertices to $\Diamondsh_l^m$ such that after deleting $(\Diamondsh_l^m)_\nu$ we still have a full rank $B$-matrix. We add for each vertex in the diamond $2$ a frozen sink. This is valid due to \cite[Lemma 2.5]{FW}.
	By Lemma \ref{L:optfr}, we can delete all frozen vertices in $\Diamondsh_l^m$ (but we keep the added frozen sinks there).
	
	Step 2. In this step we will delete all (mutable) vertices on the diamond diagonals.
	Recall that $\Diamondsh_l^m$ is obtained from $\Diamond_l^m$ by the sequence of mutations corresponding to some twists on odd diamonds.
	We apply this sequence of mutations again, then
	we get the quiver $\Diamond_l^m$ back but with some additional frozen sinks attached to the diamond $2$.
	Now we start to prove by induction from the diagonal $m$. We can delete all vertices on the diagonal $m$ by Lemma \ref{L:cuthalf}.
	After deleting all vertices on $m$, we can easily put the frozen vertices in $\Diamond_l^m$ back to the original position by sink-source extensions and mutations (see Lemma \ref{L:hiverow}).
	Then we delete all the axillary frozen sinks. 
	Suppose that we have done for the diagonal $k+1$ as shown below. 
	The dashed edges indicate that the vertices there have been deleted. 
	$$\tocutk$$
	To treat the next one -- the diagonal $k$, we first apply the flip at the (disk) diagonals $[k]$ and $\overleftarrow{[k]}$. 
	Then we are able to apply Lemma \ref{L:cuthalf} to delete vertices on the diagonal $k$.
	Finally we apply the same flip to recover the original triangulation to continue the induction.

	Step 3. After deleting all diamond diagonals, we obtain two (disconnected) copies of triangulation of $m+1$-gon (see the disk view). 
	All boundary edges except for one are dashed line. The ice quiver corresponding to such a triangulation has a full rank $B$-matrix \cite{FW}.
	It was shown in \cite[Proposition 3.4]{FW} that this rigid IQP is a polyhedral cluster model.   
	Moreover, its Jacobian algebra is finite-dimensional.
\end{proof}

\section{Computing Kronecker Coefficients} \label{S:KC}
\subsection{Recap from the previous paper} \label{ss:recap}
Suppose $\sigma$ is a weight such that $\SI_\beta(\kllm)_\sigma$ is non-zero,
then it is clear from \cite[Theorem 3]{DW1} that $\sigma(-i)\leq 0$ and $\sigma(i)\geq 0$ for any $i\in [1,l]$.
For such a weight $\sigma$ we associate two partitions 
$$\mu(\sigma):=\left(l^{-\sigma(-l)},\dots,2^{-\sigma(-2)},1^{-\sigma(-1)}\right)^*\ \text{ and }\ \nu(\sigma):=\left(l^{\sigma(l)}, \dots, 2^{\sigma(2)}, 1^{\sigma(1)}\right)^*,$$
where $*$ is the transpose of a partition. 

Recall that $W$ is the $m$-dimensional vector space spanned by the arrows from $-l$ to $l$.
Let $U$ be the subgroup of upper-triangular matrices in $\GL(W)$.
We write $\SI_\beta(\kllm)_{\sigma,\lambda}$ for the weight $\lambda$ component of $\SI_\beta(\kllm)_\sigma$.
For $\omega\in \mf{S}_m$, we write $\lambda^\omega$ for the vector in $\mb{Z}^m$ such that $\lambda^\omega(i)= \lambda(i)-i+\omega(i)$. Let $\mf{S}_m(\lambda):=\left\{\omega\in \mf{S}_m \mid \lambda^\omega\in \mb{Z}_{\geq 0}^m \right\}$.

\begin{proposition}[{\cite[Corollary 3.2 and 3.4]{Fk1}}] \label{P:KC} The dimension of $\SI_\beta(\kllm)_{\sigma,\lambda}^U$ is equal to $g_{\mu(\sigma),\nu(\sigma)}^\lambda$, which can be computed by the formula
	$$g_{\mu(\sigma),\nu(\sigma)}^\lambda = \sum_{\omega\in \mf{S}_m(\lambda)} \op{sgn}(\omega) \dim \SI_\beta(\kllm)_{\sigma,\lambda^\omega}.$$
\end{proposition}

Now Theorem \ref{T:equal} and \ref{T:CM} imply that a basis of $\SI_\beta(\kllm)$ is parametrized by $\mu$-supported $\g$-vectors in $(\br{\Diamond}_{l}^m,\br{W}_l^m)$, which is given by lattice points in some rational polyhedral cone.
We denote this cone by ${\sf G}_l^m$.
For the weight configuration $\br{\bs{\sigma}}$ on $\br{\Diamond}_l^m$ viewing as a map ${\sf G}_l^m \to \mb{R}^{n}$, we denote ${\sf G}_l^m(\sigma)$ the fibre polytope at $\sigma$.
Then it follows that each weight multiplicity $\SI_\beta(\kllm)_{\sigma,\lambda}$ is counted by the lattice points in ${\sf G}_l^m(\sigma,\lambda)$. 
\begin{theorem} \label{T:KC} Let $\sigma$ be a weight of $\kllm$ and $\lambda$ be a partition of length $\leq m$. Then
	$$g_{\mu(\sigma),\nu(\sigma)}^\lambda = \sum_{\omega\in \mf{S}_m(\lambda)} \op{sgn}(\omega) \left|\mr{G}_l^m(\sigma,\lambda^\omega)\cap \mb{Z}^{(\br{\Diamond}_l^m)_0}\right|.$$
\end{theorem}
We remain to explicitly describe the polyhedral cone $\mr{G}_l^m$.
It turns out that it is defined by dimension vectors of subrepresentations of some representations.

\subsection{The $\g$-vector cone $\mr{G}_l^m$} \label{ss:Glm}
Recall that the frozen vertices of $\Diamondsh_l^m$ are $(_{i,j}^{m\veebar})$ with $i+j=l$ if $m$ is even or $(_{0,j}^{m\veebar})$ if $m$ is odd.
For each frozen vertex $v$ we define its predecessor $v'$ to be $(_{i+1,j-1}^{m\veebar})$ if $m$ is even or $(_{0,j+1}^{m\veebar})$ if $m$ is odd.
If $v'$ is undefined (eg., $i=l-1$ for even $m$ or $j=l-1$ for odd $m$), then $v$ is called {\em terminal}.
Let $a_{v',v}$ be the unique arrow from $v'$ to $v$.

Let $\br{I}_v$ be the indecomposable injective representation of $\br{J}_l^m$ corresponding to the vertex $v$.
We consider the following set of {\em boundary representations} $T_v$ of $\br{J}_l^m$ for each frozen vertex $v$.
\begin{align}
\label{eq:injres1} & 0\to T_v \to \br{I}_v \xrightarrow{(a_{v,v'},f_v)} \br{I}_{v'}\oplus \bigoplus_{n=1}^m k_n \br{I}_n & & \text{for $v=(_{i,j}^{m\veebar})$ and $v'$ is defined};\\
\label{eq:injres2} & 0\to T_v \to \br{I}_v \xrightarrow{f_v} \bigoplus_{n=1}^m k_n \br{I}_n & & \text{for $v=(_{i,j}^{m\veebar})$ is terminal};\\
\label{eq:injres3} & 0\to T_v \to \br{I}_v \to 0 & & \text{for $v=n\in [1,m]$}.
\end{align}
\noindent Here $k_n$ are sufficiently large integers such that $T_v$ is not supported on any $n\in [1,m]$
for a general morphism $f_v:\br{I}_v \to  \bigoplus_{n=1}^m k_n \br{I}_n$. 

It is easy to see that $T_n=\br{I}_n$ ($n\in [2,m]$) is the indecomposable module (of dimension $(1,1,\cdots,1)$) supported on the (diamond) diagonal of the diamond $n$. The module $T_1=\br{I}_1$ is a ($1$-dimensional) simple module. We are going to give a concrete description of $T_v$ for $v\notin [m]$.

Let $p$ be a path $v_0 \xrightarrow{a_1} v_1 \xrightarrow{a_2} v_2 \xrightarrow{}\cdots\xrightarrow{} v_t$ in a quiver $\Delta$,
where $v_i$'s may not be all distinct.
We can attach a {\em uniserial} representation $S(p)$ of $\Delta$ with its basis elements $m_{v_i}$ labeled by $v_i$.
The action of arrows on $S(p)$ is given by $m_{v_{i-1}}a_i = m_{v_i}$ and $m_{v_j}a_i = 0$ otherwise.
\begin{proposition} \label{P:Tv} The module $T_v$ is exactly the path module $S(p_v)$ for any $v\notin [m]$.
\end{proposition}
\begin{proof}	Since $T_v$ is not supported on any frozen vertex in $[1,m]$, we see that it is in fact a representation of $(\Diamondsh_l^m,W_l^m)$, which is te restriction of $(\br{\Diamond}_l^m,\br{W}_l^m)$ to $\Diamondsh_l^m$.
Let $J_l^m:=J(\Diamondsh_l^m,W_l^m)$.
Applying the functor $\Hom_{\br{J}_l^m}(J_l^m,-)$ to \eqref{eq:injres1}, we get the exact sequence
$$0\to T_v \to I_v \xrightarrow{a_{v,v'}} I_{v'},$$
where $I_v$ is the indecomposable injective representation of $J_l^m$.
	
Let us assume that $m$ is even and $v=(_{i,j}^m)$ (other cases can be treated similarly).
Then the arrow $a_{v',v}$ is of type $C$ and the path $p_v$ is of form $e_wb^se_{0,j}^1 a^re_v$.

For any path $p$ in ${\Diamondsh}_l^m$, we denote by $[p^*]$ the equivalence class of its dual in $\big({J}_l^m\big)^*$.
Let $p$ be a maximal path such that $[p^*]\in I_v$ is in the kernel of $a_{v',v}$.
It is easy to see that the last arrow of $p$ must be of type $A$ unless $p$ is trivial.
If the second last arrow of $p$ is of type $B$, then by \eqref{eq:Jrel3} $p$ is equivalent to a zero path.
If the second last arrow of $p$ is of type $C$, then by \eqref{eq:Jrel1} $[p^*]$ does not lie in the kernel of $a_{v',v}$.
Inductively using \eqref{eq:Jrel1}, we can show that the last part of $p$ must be of the form $e_{0,j}^1 a^r e_v$. 
We claim that the previous arrow must be $_{1,j-1}^{2\vee} \to {_{0,j}^{1\vee}}={_{0,j}^1}$.
Indeed, unlike the other cases, we cannot exchange $b$ and $a$ in $e_{1,j-1}^{2\vee}bae_{1,j}^2$.  
Then the similar argument shows that the rest of $p$ must be of the form $e_wb^se_{0,j}^1$.
Throughout the proof we do not have any other choice for $p$, $p$ is in fact the maximal one,
and any path such that its dual in $\Ker a_{v,v'}$ is a subpath of $p$. 
This is easily seen to be equivalent to what we desired.   	
\end{proof}

\begin{remark} From this path description of $T_v$ we can easily verify that 
	$$\dim T_v(u) = {\sigma}_u(j)\ \text{ and }\ \dim T_{\check{v}}(u) = {\sigma}_u(-j),$$
	where $\sigma_u$ is the $\sigma$-weight of $\br{\b{s}}_l^m(u)$.
Moreover, when $l$ and $m$ are not both even, each $T_v$ can be mutated from the negative simple $S_v$ of $\bs{\mu}_v^{-1}(\Diamondsh_l^m,W_l^m)$ by $\bs{\mu}_v$,
where $\bs{\mu}_v$ is the sequence optimizing $v$.
\end{remark}

\begin{definition} A vertex $v$ is called {\em maximal} in a representation $M$
	if $\dim M(v)=1$ and all strict subrepresentations of $M$ are not supported on $v$.
\end{definition}
\noindent We note that when $v\neq v^*$, each $T_v$ contains a maximal vertex:
If $v\notin [1,m]$, the vertex $v^*$ is maximal in $T_v$. 
If $v=n$ is even (resp. odd), $(_{1,0}^n)$ (resp. $(_{l-1,0}^n)$) is maximal in $T_v$.

\begin{lemma}\cite[Lemma 6.5]{Fs1} \label{L:hom=0} Suppose that a representation $T$ contains a maximal vertex. Let $M=\Coker(\g)$, 
	then $\Hom_{\br{J}_l^m}(M,T)=0$ if and only if $\g(\dv S)\geq 0$ for all subrepresentations $S$ of $T$.
\end{lemma}

\noindent When $v$ and $v^*$ coincide, $T_v$ does not have a maximal vertex. In this case, we have an exact sequence
$0\to T_v'\to T_v\to S_v\to 0.$
Then $T_v'$ have a maximal vertex $v$.
Note that $\Hom(M,T_v')=0$ if and only if $\Hom(M,T_v)=0$ because we have another exact sequence
$0\to S\to T_v'\to U\to 0.$
In particular, the conclusion of Lemma \ref{L:hom=0} holds for this $T_v$ as well.

\begin{lemma} \label{L:TIequi} Let $M$ be any representation of $\br{J}_l^m$. We have that $\Hom_{\br{J}_l^m}(M,T_v)=0$ for each frozen $v$ if and only if $\Hom_{\br{J}_l^m}(M,I_v)=0$ for each frozen $v$.
\end{lemma}

\begin{proof}  Since each subrepresentation of $T_v$ is also a subrepresentation of $I_v$, one direction is clear.
	Conversely, let us assume that $\Hom_{\br{J}_l^m}(M,T_v)=0$ for each frozen vertex $v$.
	We prove that $\Hom_{\br{J}_l^m}(M,I_v)=0$ by induction.
	
	Recall that $T_n$ are injective for all $n\in [1,m]$.
	In general, for an injective presentation of $T$: $0\to T\to I_1\to I_0$ with $\Hom_{\br{J}_l^m}(M,I_0)=0$,
	we have that $\Hom_{\br{J}_l^m}(M,I_1)=0$ is equivalent to $\Hom_{\br{J}_l^m}(M,T)=0$.
	Now we perform the induction from a terminal $v$ then its predecessor using \eqref{eq:injres2} and \eqref{eq:injres1}.
	We can conclude that $\Hom_{\br{J}_l^m}(M,I_v)=0$ for all frozen $v$.
\end{proof}

\begin{proposition} \label{P:Glm} The cone $\mr{G}_l^m \subset \mb{R}^{(\br{\Diamond}_l^m)_0}$ is defined by $\g(\dv S)\geq 0$ for all subrepresentations $S\subseteq T_n$ for $n\in[1,m]$, and for all strict subrepresentations $S\subseteq T_v$ for all frozen $v\notin[1,m]$. 
	Moreover, any inequality above is essential.
\end{proposition}

\begin{proof} Due to Lemma \ref{L:hom=0} and \ref{L:TIequi}, $\mr{G}_l^m$ is defined by $\g(\dv S)\geq 0$ for all subrepresentations $S$ of $T_v$ and all $v$ frozen.
So it suffices to show that the condition $\g(\dv T_v)\geq 0$ is redundant for each $v\notin[1,m]$.
This is because $\dv T_v = \e_w + (\dv T_v-\e_w)$ and $\dv T_v-\e_w$ is the dimension vector of a strict subrepresentation of $T_v$.

For the last statement, we notice that each subrepresentation defining $\mr{G}_l^m$ supported on a unique frozen vertex. 
So if one defining condition is a positive combination of others, their corresponding subrepresentations must share a common frozen vertex. But this is clearly impossible.
\end{proof}

\begin{example} \label{ex:333} For $l=m=3$, the ice quiver $\br{\Diamond}_3^3$ is draw in the diamond view below.
$$\diamondthreeleft \hspace{1in} \diamondthreeright$$
We list some of the cluster variables mostly following Lemma \ref{L:wtds} ($\br{s}_{1,1}^3$ is obtained from $\wtd{s}_{1,1}^3$ by the mutation at $(_{1,1}^3)$).
Recall that $(_{0,i}^1)=(_{i,0}^1)$.
\begin{align*}
\br{s}_{i,0}^n&=s(f_{i,0}^n) && f_{i,0}^n: P_{i} \xrightarrow{n} P_{-i},\\
\br{s}_{i,j}^2&=s(f_{i,j}^2) && f_{i,j}^2: P_{i+j} \xrightarrow{\sm{2&1}} P_{-i}\oplus P_{-j}, \\
\br{s}_{1,1}^3&=s(f_{1,1}^3) && f_{1,1}^3: P_{1}\oplus P_3 \xrightarrow{\sm{0&1\\3&2}} P_{-2}\oplus P_{-2}, \\
\br{s}_{0,j}^3&=(-1)^js(f_{0,j}^3) && f_{0,j}^3: P_{j}\oplus P_3 \xrightarrow{\sm{0&1\\3&2}} P_{-j}\oplus P_{-3}.
\end{align*}

The path $p_{0,1}^3$ and $p_{0,2}^3$ are given by
\begin{align*}
 p_{0,1}^3&={_{0,1}^3}\leftarrow{_{2,1}^2}\leftarrow{_{1,1}^2}\leftarrow{_{0,1}^1}\leftarrow{_{1,0}^2}\leftarrow{_{1,1}^2}\leftarrow{_{1,2}^2}\leftarrow{_{2,0}^3}\leftarrow{_{1,1}^{3\vee}}\leftarrow{_{0,1}^{3\vee}}\\
 p_{0,2}^3&={_{0,2}^3}\leftarrow{_{1,1}^3}\leftarrow{_{1,2}^2}\leftarrow{_{0,2}^1}\leftarrow{_{1,1}^{2\vee}}\leftarrow{_{2,0}^2}\leftarrow{_{2,1}^2}\leftarrow{_{1,1}^3}\leftarrow{_{1,0}^{3\vee}}\leftarrow{_{0,2}^{3\vee}}
\end{align*}
The paths $p_{0,1}^{3\vee}$ and $p_{0,2}^{3\vee}$ are very similar.
So according to Proposition \ref{P:Glm}, the cone $\mr{G}_l^m$ have exactly $43$ supporting hyperplanes, including
$9\cdot4=36$ coming from $T_{0,1}^{3\veebar}$ and $T_{0,2}^{3\veebar}$;
$3\cdot2=6$ coming from $T_2$ and $T_3$; $1$ coming from $T_1$.

This example has been implemented in SageMath \cite{sage} to compute the relevant Kronecker coefficients.
The case when $l=m=4$ is also implemented. Readers can download the code from author's web page \cite{Fweb}.
\end{example}

As pointed out in the previous paper, the cluster structure for $\SI_\bl(\kllm)$ is not unique for $m\geq 3$.
In the end, we give a seed for $\SI_{\beta_3}(K_{3,3}^3)$, which is not mutation-equivalent to the one in the last example.
The proof is very similar.

\begin{example} \label{ex:333another}
Let $({\underline{\Diamond}_l^m}, {\underline{\b{s}}_l^m})$ be the following quiver.
$$\diamondthreeleftanother \hspace{1in} \diamondthreerightanother$$
This seed cannot be mutated to $(\br{\Diamond}_l^m, \br{\b{s}}_l^m)$ because they have different coefficients.
Each cluster variables with the same label as in Example \ref{ex:333} are equal except that
\begin{align*}
\underline{s}_{1,1}^3&=s(\underline{f}_{1,1}^3) && \underline{f}_{1,1}^3: P_{1}\oplus 2P_3 \xrightarrow{\sm{0&0&1\\2&1&0\\0&3&2}} 2P_{-2}\oplus P_{-3}, \\
\underline{s}_{0,j}^3&=(-1)^{j}s(\underline{f}_{0,j}^3) && \underline{f}_{0,j}^3: P_{j}\oplus 2P_3 \xrightarrow{\sm{0&0&1\\2&1&0\\0&3&2}} P_{-j}\oplus 2P_{-3}.
\end{align*}
\end{example}

\section*{Acknowledgement}
The author would like to thank the NCTS (National Center for Theoretical Sciences) and Shanghai Jiao Tong University for the financial support.

\bibliographystyle{amsplain}

\end{document}